\documentclass[10pt]{article}
\usepackage{amsfonts}
\usepackage{}
\usepackage[notcite,notref]{} 
\usepackage{blindtext}
\usepackage[a4paper, total={6.6in, 8.7in}]{geometry}
\usepackage{amsthm}
\usepackage{CJK}
\usepackage{graphicx}
\usepackage{amsmath}
\usepackage{setspace}
\usepackage{xcolor}
\usepackage{mathtools}
\usepackage{blkarray}

\theoremstyle{definition}
\newtheorem{theorem}{Theorem}
\newtheorem{lemma}{Lemma}

\newtheorem{definition}{Definition}

\newtheorem{example}{Example}

\newcounter{nex}
\usepackage[bitstream-charter]{mathdesign}
\usepackage[T1]{fontenc}
\usepackage{bbm}

\newcommand{\CC}{\mathbb{C}}

\DeclareMathOperator{\ind}{ind}

\DeclareMathOperator{\matt}{\bf mat}
\DeclareMathOperator{\diag}{diag}

\DeclareMathOperator{\squeeze}{squeeze}

\usepackage[english, algosection, algoruled, noline]{algorithm2e}

\usepackage{eucal}
\newcommand{\ac}{\mathcal{P}}
\newcommand{\bc}{\mathcal{B}}
\newcommand{\cc}{\mathcal{C}}
\newcommand{\dc}{\mathcal{D}}

\newcommand{\gc}{\mathcal{G}}
\newcommand{\hc}{\mathcal{H}}
\newcommand{\ic}{\mathcal{I}}

\newcommand{\pc}{\mathcal{Y}}

\newcommand{\uc}{\mathcal{U}}
\newcommand{\vc}{\mathcal{V}}

\newcommand{\xc}{\mathcal{X}}
\newcommand{\yc}{\mathcal{Y}}
\newcommand{\zc}{\mathcal{Z}}
\newcommand{\wc}{\mathcal{W}}

\begin{document}
\begin{spacing}{1.1}

\title{A study on the 1-$\Gamma$ inverse of  tensors via the M-Product}

\author{
Siran Chen,\thanks{School of Mathematical Sciences, Guangxi Minzu University, 530006,
Nanning, PR China. { E-mail address}:  18877545298@163.com}\quad
Hongwei Jin,\thanks{School of Mathematical Sciences, Guangxi Minzu University, 530006, Nanning, PR China; Center for Applied Mathematics of Guangxi, Guangxi Minzu University, Nanning, 530006, PR China. { E-mail address}:  jhw$\_$math@126.com} \quad
Shaowu Huang,\thanks{Corresponding author. School of Mathematics and Finance, Putian University, Putian, PR China. { E-mail address}:  shaowu2050@126.com}
\quad
Julio Ben\'{\i}tez
	\thanks{Departamento de Matem\'atica Aplicada,
	Instituto de Matem\'atica Multidisciplinar,
	Universitat Polit\'{e}cnica de Val\'{e}ncia, Camino de Vera s/n,
	46022, Valencia, Spain.
	{E-mail addresses}: jbenitez@mat.upv.es }}

\date{}

\maketitle

\begin{abstract}
 In this paper, we will study the issue about the 1-$\Gamma$ inverse, where $\Gamma\in\{\dag,\ D,\ *\}$,   via the M-product. The aim of the current study is threefold. Firstly, the definition  and characteristic of the 1-$\Gamma$ inverse is introduced. Equivalent conditions for a tensor to be a 1-$\Gamma$ inverse are established. Secondly, using the singular value decomposition, the corresponding numerical algorithms for computing the 1-$\Gamma$ inverse are given.
Finally, the solutions of the multilinear equations related 1-$\Gamma$ inverse are studied, and numerical calculations are given to verify our conclusions.~\\


\end{abstract}


\section{Introduction}
A tensor is a multidimensional array and a first-order tensor is a vector, a second-order tensor is a matrix, and tensors of third or higher order are called higher-order tensors. A tensor takes the form
$\ac=\left(\ac_{{i_1}i_2\ldots{i_m}}\right)\in\mathbb{C}^{{n_1}\times{n_2}\times\ldots{n_m}}$ with entries over the complex field.  The positive integer $m$ is the order of the tensor $\ac$.
Higher-order tensors have been used in various fields, such as robust tensor PCA {\rm \cite{kong1,L3}}, low-rank tensor recovery {\rm \cite{Che,QW}}, image processing {\rm \cite{KBHH,MSL,Soltani,Tarzanagh}}, computer vision {\rm \cite{4,5}}, signal processing {\rm \cite{1,RU,2,3}}, face recognition {\rm \cite{KBHH,Chen}}, date completion and denoising {\rm \cite{Hu,Long,A2}} and so on.


Tensor multiplication is a basic and critical operation similar to matrix multiplication, which has attracted considerable attention in various scientific disciplines.
Kilmer {\it et al.} {\rm \cite{K1}} first proposed the T-product.
Jin {\it et al.} {\rm \cite{J1}} proposed the cosine change product.
Kernfeld {\it et al.} {\rm \cite{KK}} proposed the M-product, and Einstein {\rm \cite{A1}} proposed the Einstein product.
Shao {\rm \cite{Shao1}} proposed the general product.
The $n$-mode product has been proposed by Qi {\rm \cite{L1}}.
Kilmer {\it et al.} {\rm \cite{Perrone}} introduced a new form of tensor multiplication that allows the representation of third-order tensors as products of other third-order tensors.
Lund {\rm \cite{Lund}} gave a definition of T-functions based on the T-product of the third-order F-square tensor. T-functions were used for fast deep learning in stable tensor neural networks {\rm \cite{Newman}}.

Miao {\it et al.} {\rm \cite{Miao2}} studied the tensor similar relationship and then investigated the T-Jordan canonical form based on the T-product of a tensor. Meanwhile, the T-QR, T-LU, T-polar and T-Schur decompositions of tensors were established. Besides, the T-group inverse and T-Drazin inverse were also studied.
Sun {\it et al.} defined the \{i\}-inverse and group inverse based on  the general product of the tensor, and studied the properties of the generalized inverse of the tensor, and also defined the Moore-Penrose inverse of the tensor by Einstein product, and obtained the explicit representation of the Moore-Penrose inverse of the tensor block {\rm \cite{Sun}}.
Ji {\it et al.} {\rm \cite{Ji2}} generalized the concept of the Drazin inverse of a square matrix to even order square tensor. The expression of the Drazin inverse was obtained by the kernel-nilpotent decomposition.
Behera {\it et al.} {\rm \cite{Nandi}} further elaborated the Drazin inverse and W-weighted Drazin inverse theory of tensors. In addition, different types of methods are established to compute the Drazin inverse of tensors.
Liang {\it et al.} \cite{Liang} generalized the concept of  the MP inverse of matrices to further results of the Einstein product in the case of tensors and got its application to tensor approximation problems.
The authors of \cite{Che1} studied the  calculation of a complete orthogonal decomposition of a third-order tensor, known as the T-URV decomposition.
Cong {\it et al.} in \cite{Cong} showed that the Moore-inverse of the tensor can be represented by T-SVD with T-product, and gave the equivalent conditions for the stable perturbation of the MP inverse.
Further results on the generalized inverse of tensors on the Einstein product were investigated in \cite{Behera}.
Cong {\it et al.} studied the characterization and perturbation of the EP inverse of the tensor kernel based on the T-product \cite{Cong2}.
The core and core-EP inverses of tensors was studied in \cite{Sahoo}.

One of the crucial operations for handling data and information in multidimensional space is the M-product of tensors.
Jin {\it et al.} \cite{JXWL} studied the MP inverse of tensors with M-product.
They also gave the conditions for understanding it as the least square solution or minimum norm solution.
The authors of \cite{SPBS} introduced the concepts of the Drazin inverse and core-EP inverse of tensors by  the M-product and had a further study on the CMP, DMP, and MPD inverses of tensors.
Panigrahy {\it et al.} \cite{Panigrahy} proposed a new full-rank decomposition method based on M-QDR decomposition of the third-order tensor of the M-product.

In this paper, the tensor 1-$\Gamma$ inverse extending through the M-product will be examined. This article is organized as follows. The words and symbols required for this work are initially provided in the second section, after which we introduce the M-product between two tensors. In the third section, we provide the 1-MP inverse under the M-product and provide the equivalent expression for the 1-MP inverse by using the singular value decomposition. After that, we go over a rough description of the 1-MP inverse.  Simultaneously, we examine the tensor 1-D inverse and provide the associated numerical algorithm for computing it. Additionally, we provide a representation of this component tensor inverse. Lastly, the tensor is 1-Star inverse under the M-product is explained. The corresponding attributes are provided along with the establishment of the 1-Star inverse numerical calculation method. The solutions to the multilinear equations concerning the 1-MP, 1-D, and 1-star inverses are provided in section 4. Numerical examples validate the theoretical results.

\section{Preliminaries}\label{two}

Following standard notation, matrices and higher order tensors will be denoted by capital letters and by capital, caligraphical letters like $P,~ \ac$, respectively.
Moreover, $\ac\left(i,:,:\right)$,~$\ac\left(:,i,:\right)$ and $\ac\left(:, :, i\right)$ correspond to $i$th horizontal slice, lateral slice and frontal slice, respectively.
For simplicity, to the $i$th frontal slice will be denoted by $\ac^{(i)}$. The notation $i\in\left[n\right]$ will be used to denote the index range $i=1,\ldots,n$.

We begin this part with the face-wise product of two tensors.

\begin{definition}{\rm \cite{KK}}
The {\bf face-wise product} $\cc\triangle\dc$ of $\cc \in \mathbb{C}^{n_1\times n_2\times n_3}$ and $\dc \in \mathbb{C}^{n_2 \times l\times n_3}$ is defined as
$$\left(\cc\triangle\dc\right)^{\left(i\right)}=\cc^{\left(i\right)}\dc^{\left(i\right)}, \ i\in\left[n_3\right],$$
where $\left(\cc\triangle\dc\right)^{\left(i\right)}$, $\cc^{\left(i\right)}$, $\dc^{\left(i\right)}$ denote the $i$th frontal slice of $\cc\triangle\dc$, $\cc$ and $\dc$, respectively.
\end{definition}



The next definition corresponds to the $k$-mode product of a tensor with a matrix.

\begin{definition}{\rm \cite{KB}}
The {\bf $k$-mode product}  of a tensor $\cc \in \mathbb{C}^{n_1\times n_2\times\cdots\times n_p}$ with a matrix $N\in\mathbb{C}^{J\times n_k}$ is denoted by $\cc\times_kN$
and expresses as
\begin{equation*}
  \left(\cc\times_kN\right)_{i_1 i_2\ldots i_{k-1}ji_{k+1}\ldots i_p}=\sum\limits_{i_k=1}^{n_k}\mathcal{{C}}_{i_1i_2\ldots i_p}N_{ji_k}, \ j\in\left[J\right].
\end{equation*}
\end{definition}
Denote
$L\left(\cc\right)=\cc\times_{3}{M},~L^{-1}\left(\cc\right)=\cc\times_{3}{M}^{-1},
$
where ${M}\in\mathbb{C}^{n_3\times n_3}$ is a nonsingular matrix. Then, one has the following product.

\begin{definition}{\rm \cite{KK, K3}}
Let ${M}\in\mathbb{C}^{n_3\times n_3}$ be a nonsingular matrix. The {\bf $\star_{M}$ product} between $\cc \in \mathbb{C}^{n_1\times n_2\times n_3}$ and $\dc \in \mathbb{C}^{n_2 \times l\times n_3}$ is defined by
\begin{equation}\label{XZ}
\cc\star_{M}\dc=L^{-1}\left[L\left(\cc\right)\triangle L\left(\dc\right)\right].
\end{equation}
\end{definition}
The product defined in \eqref{XZ} is known as the "M-product" between two third-order tensors. In fact, $\cc\times_{3}M$ is computed by using the matrix-matrix product
\begin{equation}
  \mathcal{Y}=\cc\times_3M\Leftrightarrow\mathcal{Y}_{\left(3\right)}
  =M\cc_{\left(3\right)},
\end{equation}
where $\cc_{\left(3\right)}\in\mathbb{C}^{n_3\times n_1n_2}$ denotes the mode-3 unfolding of $\cc$, which is got by the \textbf{squeeze} function proposed in \cite{KBHH}.
More precisely,
\begin{equation}\label{as}
 \cc_{(3)}=\left[\left(\squeeze\left(\overrightarrow{\cc}_1\right)\right)^T, \left(\squeeze\left(\overrightarrow{\cc}_2\right)\right)^T\ldots, \left(\squeeze\left(\overrightarrow{\cc}_{n_2}\right)\right)^T\right],
\end{equation}
where $\overrightarrow{\cc}_j$, $j\in\left\{1,\ldots,n_2\right\}$, are the lateral slices of $\cc$ and $\squeeze\left(\cdot\right)$ is defined by
\begin{equation*}
  C_j=\squeeze\left(\overrightarrow{\cc}_j\right)\Rightarrow\left(C_j\right)_{ik}=\left(\overrightarrow{\cc}_j\right)_{i1k}, \ i\in\left[n_1\right], \ j\in\left[n_2\right], \   k\in\left[n_3\right].
\end{equation*}
So,  it may be desirable to have  an alternative expression for (\ref{XZ}).
More precisely,
\begin{equation}\label{nb}
\cc\star_M\dc=L^{-1}\left[L\left(\cc\right)\triangle L\left(\dc\right)\right]=\left[L\left(\cc\right)\triangle L\left(\dc\right)\right]\times_3{M}^{-1}=
\left[\left(\cc\times_3{M}\right)\triangle\left(\dc\times_3{M}\right)\right]\times_3{M}^{-1}.
\end{equation}
Notice that one can get the T-product when choosing ${M}$ as the normalized DFT matrix in (\ref{nb}) (see  \cite{K1}). Denote
\begin{equation}\label{vvb}
\matt\left(\cc\right)=\diag\left(L\left(\cc\right)^{\left(1\right)},\ldots,L\left(\cc\right)^{\left(n_3\right)}\right)
=\diag\left(\widehat{\cc}^{\left(1\right)},\ldots,\widehat{\cc}^{\left(n_3\right)}\right),
\end{equation}
and $\matt^{-1}\left(\cdot\right)$ is the inverse operation of $\matt\left(\cdot\right)$, that is
                       \begin{equation}\label{mmm}
                         \cc=\matt^{-1}\left[\matt\left(\cc\right)\right].
\end{equation}
 Then, one can have the alternative definition of the M-product.
\begin{definition}{\rm \cite{SPBS}}
Let ${M}\in\mathbb{C}^{n_3\times n_3}$ be a nonsingular matrix.  The $\star_M$ product over arguments $\cc \in \mathbb{C}^{n_1\times n_2\times n_3}$ and $\dc \in \mathbb{C}^{n_2 \times l\times n_3}$ is defined by
\begin{equation}\label{zzz}
\cc\star_M\dc=\matt^{-1}\left[\matt\left(\cc\right)\cdot\matt\left(\dc\right)\right].
\end{equation}
\end{definition}

In the following, we investigate elementary operations based on the M-Product.

\begin{lemma}{\rm \cite{K3,KK}} 
If $\ac, \gc, \hc$ are third-order tensors with proper sizes, then the subsequent statements are valid:
\begin{description}
  \item  [(a)] $\ac\star_M\left(\gc+\hc\right) = \ac\star_M\gc + \ac\star_M\hc$.
\item [(b)] $\left(\gc+\hc\right)\star_M\ac = \gc\star_M\ac + \hc\star_M \ac$.
\item [(c)] $\left(\ac\star_M\gc\right)\star_M\hc = \ac\star_M\left(\gc\star_M\hc\right)$.
\end{description}
\end{lemma}

\begin{definition} {\rm \cite{KK}}
The  {\bf identity tensor} $\ic \in \mathbb{C}^{n_1\times n_1\times n_3}$ is determined by $L\left(\ic\right)^{\left(i\right)}={I}_n$, $i\in \left[n_3\right]$.
\end{definition}

According to the definition of $L\left(\ic\right)^{\left(i\right)}$, it can be concluded
$$L\left(\ic\right)\triangle L\left(\ac\right)=L\left(\ac\right)\triangle L\left(\ic\right)=L\left(\ac\right),$$
which implies
$$\ac\star_M\ic=\ic\star_M\ac=\ac.$$

\begin{definition}{\rm \cite{KK}}
The tensor $\ac \in \mathbb{C}^{n_1 \times n_1\times n_3}$ is invertible if there exists $\xc \in \mathbb{C}^{n_1 \times n_1\times n_3}$ such that
$$\ac\star_M\xc=\xc\star_M\ac=\ic.$$
In such conditions, $\xc$ represents the inverse $\ac^{-1}$ of $\ac$, and $\xc$ is unique.
\end{definition}

In the following, we give the definition of the conjugate transpose of a tensor $\ac$.

\begin{definition}{\rm \cite{KK}}
The {\bf conjugate-transpose} of $\ac \in \mathbb{C}^{n_1 \times n_2 \times n_3}$, marked with $\ac^*$, is the $n_2 \times n_1 \times n_3$ tensor defined by $L\left(\ac^*\right)^{\left(i\right)}=\left[L\left(\ac\right)^{\left(i\right)}\right]^*$, $i\in\left[n_3\right]$.
\end{definition}

\begin{lemma} {\rm \cite{K3,KK}}\label{ppaa}
Arbitrary tensors $\cc \in \mathbb{C}^{n_1\times n_2\times n_3}$ and $\dc \in \mathbb{C}^{n_2 \times l \times n_3}$ satisfy $\left(\cc\star_M \dc\right)^*=\dc^*\star_M\cc^*$.
\end{lemma}

\begin{definition} 
For $\ac \in \CC^{n_1\times n_1\times n_3}$, suppose
$$ 
\matt\left(\ac\right)=\diag\left(L\left(\ac\right)^{\left(1\right)},\ldots,L\left(\ac\right)^{\left(n_3\right)}\right).
$$
Then, the index of $\ac$, which is denoted by $\ind\left(\ac\right)$, is
$\ind\left(\ac\right)=\max\limits_{i\in\left[n_3\right]}\left\{\ind\left(L\left(\ac\right)^{\left(i\right)}\right)\right\}$.
\end{definition}



\section{Main results}\label{three0}

In this section, we will study the $1$-$\Gamma$ inverse, where $\Gamma\in\left\{\dag, D, *\right\}$, that is, 1-MP inverse, 1-D inverse and 1-Star inverse.

\subsection{The 1-MP inverse of a tensor}

In the following, we will give the definition of the 1-MP inverse. Before that, we firstly give the definition of the Moore-Penrose inverse of a tensor $\ac$.

\begin{definition}\cite{JXWL}
Let $\ac \in \CC^{n_1\times n_2\times n_3}$. If there exists a tensor $\wc \in
\CC^{n_2\times n_1\times n_3}$ such that
\begin{equation}\label{mpkk0}
\textnormal{(I)} \ \ac\star_M\wc\star_M\ac = \ac, \quad \textnormal{(II)} \ \wc\star_M\ac\star_M\wc = \wc, \quad \textnormal{(III)} \ \left(\ac\star_M\wc\right)^* = \ac\star_M\wc, \quad \textnormal{(IV)} \ \left(\wc\star_M\ac\right)^* = \wc\star_M\ac,
\end{equation}
then $\wc$ is called the \textbf{Moore-Penrose inverse} of the tensor $\ac$ and is denoted by $\ac^\dag$, and $\wc$ is unique.
\end{definition}

For any $\ac \in \CC^{n_1\times n_2\times n_3}$, denote by $\ac\left\{i,j,k,l\right\}$, the set of all $\ac$  which satisfy equations (I), (II), (III), (IV) of (\ref{mpkk0}). In this case, $\wc$  is a $\left\{i,j,k,l\right\}$-inverse. Especially, one of the $\{1\}$-inverse of $\ac$ will be denoted by $\ac^-$.

\begin{definition}\label{d07}
Let $\ac \in\mathbb{C}^{n_1\times n_2\times n_3}$. For each $\ac^-\in
\ac\{1\}$, the tensor
\begin{equation}\label{mp012}
\ac^{-,\dag}=\ac^-\star_M\ac\star_M\ac^\dag
\end{equation}
is called the \textbf{1-MP inverse} of the tensor $\ac$ and is denoted by $\ac^{-,\dag}$.
The symbol $\ac\{-,\dag\}$ stands for the set of all 1-MP inverses of $\ac$.
\end{definition}
 Let the singular value decomposition of $\ac$ is
\begin{align}\label{333}
  \ac=\uc\star_M\matt^{-1}\left[\diag\left(\left[
                                \begin{array}{cc}
                                  D^{\left(1\right)} & O \\
                                  O & O \\
                                \end{array}
                              \right],\ldots,\left[
                                \begin{array}{cc}
                                  D^{\left(n_3\right)} & O \\
                                  O & O \\
                                \end{array}
                              \right]
  \right) \right]\star_M\vc^*.
\end{align}
Then, the general form of the $\{1\}$-inverses of $\ac$ is given by
\begin{align}\label{555}
  \ac^-=\vc\star_M\matt^{-1}\left[\diag\left(\left[
                                \begin{array}{cc}
                                  \left(D^{\left(1\right)}\right)^{-1} & W^{\left(1\right)}_{12} \\
                                  W^{\left(1\right)}_{21} & W^{\left(1\right)}_{22} \\
                                \end{array}
                              \right],\ldots,\left[
                                \begin{array}{cc}
                                  \left(D^{\left(n_3\right)}\right)^{-1} & W^{\left(n_3\right)}_{12} \\
                                  W^{\left(n_3\right)}_{21} & W^{\left(n_3\right)}_{22} \\
                                \end{array}
                              \right]
 \right)\right]\star_M\uc^*,
\end{align}
where $W^{ \left(i\right)}_{12}$, $W^{\left(i\right)}_{21}$, $W^{\left(i\right)}_{22}$, $i\in\left[n_3\right]$ are arbitrary. In addition, one can have
\begin{eqnarray}\label{0xxc}
  \ac\left\{-,\dag\right\}&=&\left\{\ac^-\star_M\ac\star_M\ac^{\dag}|\ac^-\in\ac\{1\}\right\}\nonumber\\
  &=&\left\{\vc\star_M\matt^{-1}\left[\diag\left(\left[
                                \begin{array}{cc}
                                  \left(D^{\left(1\right)}\right)^{-1} & O \\
                                  W^{\left(1\right)}_{21} & O \\
                                \end{array}
                              \right],\ldots,\left[
                                \begin{array}{cc}
                                  \left(D^{\left(n_3\right)}\right)^{-1} & O \\
                                  W^{\left(n_3\right)}_{21} & O \\
                                \end{array}
                              \right]
 \right ) \right]\star_M\uc^*\right\}.
\end{eqnarray}

 Notice that the 1-MP inverses is not unique. We use $\mathcal{P}\{-,\dag\}=\left\{\ac^{-,\dag}|\ac \in\mathbb{C}^{n_1\times n_2\times n_3}\right\}$ to denote the set of the 1-MP inverses of a tensor.

An algorithm  for computing the 1-MP inverse based on (\ref{0xxc}) follows. Notice that $L\left(\alpha\right)=\alpha\times_3 M$, $\alpha=\ac, \bc,\cc\ldots$, will denoted by $\widehat{\alpha}$ from now on.

\begin{algorithm}[H]\label{a1}
\caption{Computing the 1-MP inverse under the M-product}	
\KwIn{$\ac\in\mathbb{C}^{n_1\times n_2\times n_3}$ and $M \in\mathbb{C}^{n_3\times n_3}$}
\KwOut {$\xc=\ac^{-,\dag}$}
\begin{enumerate}
\item Compute $\widehat{\ac}=\ac\times_3M$
\item \textbf{For} $i=1:n_3$ \textbf{do}
\item $\widehat{\ac}^{\left(i\right)}=\widehat{\uc}^{\left(i\right)}\cdot\widehat{\mathcal {S}}^{\left(i\right)}\cdot\left(\widehat{\mathcal {V}}^{\left(i\right)}\right)^*$, that is, the singular value decompositions of $\widehat{\ac}^{\left(i\right)}$
\item $\diag\left(\widehat{D}^{\left(i\right)},O\right)=\widehat{\mathcal {S}}^{\left(i\right)}$,
\item $\widehat{\ac^-}^{\left(i\right)}=\widehat{\vc}^{\left(i\right)}
\begin{bmatrix}
\left(\widehat{D^{\left(i\right)}}\right)^{-1}            &  \widehat{W_{12}}^{\left(i\right)}\\
\widehat{W_{21}}^{\left(i\right)}              &  \widehat{W_{22}}^{\left(i\right)}
\end{bmatrix}\left(\widehat{\uc}^{\left(i\right)}\right)^*,
$ where $\widehat{W_{12}}^{\left(i\right)}$, $\widehat{W_{21}}^{\left(i\right)}$, $\widehat{W_{22}}^{\left(i\right)}$ are arbitrary
\item $\left(\widehat{\ac^{-,\dag}}\right)^{\left(i\right)}=\widehat{\vc}^{\left(i\right)}
\begin{bmatrix}
\left(\widehat{D^{\left(i\right)}}\right)^{-1}         &  O\\
\widehat{W_{21}}^{\left(i\right)}           &  O
\end{bmatrix}\left(\widehat{\uc}^{\left(i\right)}\right)^*
$
\item $\zc^{\left(i\right)}=\left(\widehat{\ac^{-,\dag}}\right)^{\left(i\right)}$
\item \textbf{End of}
\item Compute $\xc=\zc\times_3M^{-1}$
\item \textbf{Return} $\ac^{-,\dag}=\xc$.
\end{enumerate}
\end{algorithm}

\begin{example}
Let $\ac\in\mathbb{C}^{3\times2\times3}$ and $M\in\mathbb{C}^{3\times3}$  with entries
\footnotesize{$$\mathcal P^{\left(1\right)}=
\begin{bmatrix}
1 & -2\\
0 & -6\\
0 & 0
\end{bmatrix},~
\mathcal P^{\left(2\right)}=
\begin{bmatrix}
1 & 2 \\
2 & 4 \\
0 & 0
\end{bmatrix},~
\mathcal P^{\left(3\right)}=
\begin{bmatrix}
0 & 3 \\
0 & 6 \\
0 & 0
\end{bmatrix},~
M=
\begin{bmatrix}
1 & 0 & 1\\
0 & 1 & 0 \\
0 & 1 & 1
\end{bmatrix}.
$$
}\normalsize{By Algorithm \ref{a1}, we firstly compute $\widehat{\ac^{-}}^{\left(1\right)},\ \widehat{\ac^{-}}^{\left(2\right)},\  \widehat{\ac^{-}}^{\left(3\right)}$ as}
\footnotesize{\begin{eqnarray*}
\widehat{\ac^{-}}^{\left(1\right)}&=&
\begin{bmatrix}
0.7071 & -0.7071 \\
0.7071 & 0.7071
\end{bmatrix}
\begin{bmatrix}
0.7071 & a_{12} & a_{13}\\
a_{21} & a_{22} & a_{23}
\end{bmatrix}
\begin{bmatrix}
1 & 0 & 0 \\
0 & 1 & 0 \\
0 & 0 & 1
\end{bmatrix}\\
&=&\begin{bmatrix}
0.5-0.7071a_{21} & 0.7071a_{12}-0.7071a_{22} & 0.7071a_{13}-0.7071a_{23}\\
0.5+0.7071a_{21} & 0.7071a_{12}+0.7071a_{22} & 0.7071a_{13}+0.7071a_{23}
\end{bmatrix},
\end{eqnarray*}
\begin{eqnarray*}
\widehat{\ac^{-}}^{\left(2\right)}&=&
\begin{bmatrix}
-0.4472 & -0.8944 \\
-0.8944 & 0.4472
\end{bmatrix}
\begin{bmatrix}
0.2    & b_{12} & b_{13}\\
b_{21} & b_{22} & b_{23}
\end{bmatrix}
\begin{bmatrix}
-0.4472 & -0.8944 & 0 \\
-0.8944 & 0.4472  & 0 \\
0       & 0       & 1
\end{bmatrix}\\
&=&
\begin{bmatrix}
0.04+0.4b_{21}+0.4b_{12}+0.8b_{22}   & 0.08+0.8b_{21}-0.2b_{12}-0.4b_{22} &-0.4472b_{13}-0.8944b_{23}\\
0.08-0.2b_{21}+0.8b_{12}-0.4b_{22} & 0.16-0.4b_{21}-0.4b_{12}+0.2b_{22} & -0.8944b_{12}+0.4472b_{23}
\end{bmatrix},
\end{eqnarray*}
\begin{eqnarray*}
\widehat{\ac^{-}}^{\left(3\right)}&=&
\begin{bmatrix}
-0.1961 & -0.9806 \\
-0.8944 & 0.1961
\end{bmatrix}
\begin{bmatrix}
0.0877    & c_{12} & c_{13}\\
c_{21}    & c_{22} & c_{23}
\end{bmatrix}
\begin{bmatrix}
-0.4472 & -0.8944 & 0 \\
-0.8944 & 0.4472  & 0 \\
0       & 0       & 1
\end{bmatrix}\\&=&
 \begin{bmatrix}
 0.0077+0.4385c_{21}+0.1754c_{12}+0.877c_{22}  & 0.0154+0.877c_{21}-0.877c_{12}-0.4385c_{22}  &-0.1961c_{13}-0.9806c_{23}\\
0.0351-0.0877c_{21}+0.8c_{12}-0.1754c_{22} & 0.0701-0.1754c_{21}-0.4c_{12}+0.0877c_{22} & -0.8944c_{13}+0.1961c_{23}
\end{bmatrix}.
\end{eqnarray*}}
\normalsize{Then, the frontal slices of $\ac^{-,\dag}$ are}
\footnotesize{$$
\left(\ac^{-,\dag}\right)^{\left(1\right)}=
\begin{bmatrix}
0.4b_{21}-0.7071a_{21}-0.4385c_{21}+0.5323     &  0.8b_{21}-0.8771c_{21}+0.0646  & 0\\
0.7071a_{21}-0.2b_{21}+0.0877c_{21}+0.5415  &     0.1754c_{21}-0.4b_{21}+0.0831         & 0
\end{bmatrix},
$$
$$
\left(\ac^{-,\dag}\right)^{\left(2\right)}=
\begin{bmatrix}
0.4b_{21}+0.04    & 0.8b_{21}+0.08 & 0\\
0.08-0.2b_{21} & 0.16-0.4b_{21} & 0
\end{bmatrix},$$
$$\left(\ac^{-,\dag}\right)^{\left(3\right)}=
\begin{bmatrix}
0.4385c_{21}-0.4b_{21}-0.0323      & 0.8771c_{21}-0.8b_{21}-0.0646 & 0\\
0.2b_{21}-0.0877c_{21}-0.0415    & 0.4b_{21}-0.1754c_{21}-0.0831           & 0
\end{bmatrix}.
$$
}\end{example}

In the following, we will obtain the characteristics of the 1-MP inverse of a tensor from several angles.
\begin{lemma}\label{111}
Let $\ac\in\mathbb{C}^{n_1\times n_2\times n_3}$ and $\ac^-\in
\ac\left\{1\right\}$. Then,
\begin{description}
  \item  [(a)] $\ac^{-,\dag}\in\ac\left\{1,2,3\right\}$.
  \item  [(b)] $\ac^{-,\dag}\star_M\ac=\ac^{-}\star_M\ac$ and  $\ac\star_M\ac^{-,\dag}=\ac\star_M\ac^{\dag}$.
\end{description}
\begin{proof} {\rm\bf(a)} Since $\ac^{-},~\ac^\dag \in \ac\left\{1\right\}$, it follows that $\ac^{-,\dag}\in\ac\left\{1,2\right\}$. On the other hand,
$$\ac\star_M\ac^{-,\dag}=\ac\star_M\ac^{-}\star_M\ac\star_M\ac^{\dag}=\ac\star_M\ac^{\dag},$$
we have $\ac^{-,\dag}\in \ac\left\{3\right\}$. So, $\ac^{-,\dag}\in \ac\left\{1,2,3\right\}$.

{\rm\bf(b)}
$\ac^{-,\dag}\star_M\ac=\ac^{-}\star_M\ac\star_M\ac^\dag\star_M\ac=\ac^-\star_M\ac$. The other equality was proved in the previous item.
\end{proof}
\end{lemma}

\begin{lemma}\label{112}
Let $\ac\in\mathbb{C}^{n_1\times n_2\times n_3}$ and $\mathcal {G}$ be a fixed $\{1\}$-inverse of $\ac$. Then, the class of all $\left\{1\right\}$-inverses of $\ac$ is given by
$$\ac\left\{1\right\}=\left\{\mathcal {G}+\mathcal {U}-\mathcal {G}\star_M\ac\star_M\mathcal {U}\star_M\ac\star_M\mathcal {G},~\text{where}~\mathcal{U}~\text{is arbitrary}\right\}.$$
\begin{proof}
First of all, notice that for all $\mathcal {U}$,
\begin{align*}
 \ac\star_M\left(\mathcal {G}+\mathcal {U}-\mathcal {G}\star_M\ac\star_M\mathcal {U}\star_M\ac\star_M\mathcal {G}\right)\star_M\ac &=\ac\star_M\mathcal {G}\star_M\ac+\ac\star_M\mathcal {U}\star_M\ac-\ac\star_M\mathcal {G}\star_M\ac\star_M\mathcal {U}\star_M\ac\star_M\mathcal {G}\star_M\ac\\
 &=\ac\star_M\mathcal {G}\star_M\ac\\
 &=\ac.
\end{align*}
So, $\mathcal {G}+\mathcal {U}-\mathcal {G}\star_M\ac\star_M\mathcal {U}\star_M\ac\star_M\mathcal {G}$ is a $\left\{1\right\}$-inverse of $\ac$ for all $\mathcal {U}$.

Suppose $\mathcal {G}_1$ is another $\left\{1\right\}$-inverse of $\ac$. Set $\mathcal {U}=\mathcal {G}_1-\mathcal {G}$ and check that $\mathcal {G}_1=\mathcal {G}+\mathcal {U}-\mathcal {G}\star_M\ac\star_M\mathcal {U}\star_M\ac\star_M\mathcal {G}$. Thus, the class of all $\left\{1\right\}$-inverse of $\ac$ is given by
$$\ac\left\{1\right\}=\left\{\mathcal {G}+\mathcal {U}-\mathcal {G}\star_M\ac\star_M\mathcal {U}\star_M\ac\star_M\mathcal {G},~\text{where}~\mathcal {U}~\text{is arbitrary}\right\}.$$
\end{proof}
\end{lemma}

By Lemma \ref{111}, we have that $\ac \subseteq \ac\left\{-,\dag\right\}$ holds. In addition by lemma \ref{111} every tensor $\xc \in \ac\left\{-,\dag\right\}$ satisfies the two equations of the system given by $\xc\star_M\ac\star_M\xc=\xc$, $\ac\star_M\xc=\ac\star_M\ac^\dag$.

\begin{theorem} \label{113}
Let $\ac\in\mathbb{C}^{n_1\times n_2\times n_3}$. Then, the following statements are equivalent.
\begin{description}
  \item  [(a)] $\wc\in\ac\left\{-,\dag\right\}$.
  \item  [(b)] $\wc$ is a solution of the system: $\xc\star_M\ac\star_M\xc=\xc$, $\ac\star_M\xc=\ac\star_M\ac^\dag$.
   \item  [(c)] $\wc\in\ac\left\{1,2,3\right\}$.
\end{description}
\begin{proof}
$\textbf{(a)}\Longrightarrow \textbf{(b)}$: Suppose  that $\wc\in\ac\left\{-,\dag\right\}$, then $\wc=\ac^{-}\star_M\ac\star_M\ac^\dag$ for some $\ac^-\in\ac\left\{1\right\}$. Thus, we have
\begin{align*}
 \wc\star_M\ac\star_M\wc&=\ac^-\star_M\ac\star_M\ac^\dag\star_M\ac\star_M\ac^-\star_M\ac\star_M\ac^\dag\\
 &=\ac^-\star_M\ac\star_M\ac^\dag\star_M\ac\star_M\ac^\dag\\
 &=\ac^-\star_M\ac\star_M\ac^\dag=\wc
\end{align*}
and
\begin{align*}
\ac\star_M\wc=\ac\star_M\ac^-\star_M\ac\star_M\ac^\dag=\ac\star_M\ac^\dag.
\end{align*}
$\textbf{(b)}\Longrightarrow \textbf{(a)}$: Assume that $\wc$ satisfies $\xc\star_M\ac\star_M\xc=\xc$ and $\ac\star_M\xc=\ac\star_M\ac^\dag$ and let $\ac$ be written in (\ref{333}),
\begin{align*}
  \ac=\uc\star_M\matt^{-1}\left[\diag\left(\left[
                                \begin{array}{cc}
                                  D^{\left(1\right)} & O \\
                                  O & O \\
                                \end{array}
                              \right],\ldots,\left[
                                \begin{array}{cc}
                                  D^{\left(n_3\right)} & O \\
                                  O & O \\
                                \end{array}
                              \right]
  \right)\right]\star_M\vc^*.
\end{align*}
Let
\begin{align*}
  \wc=\vc\star_M\matt^{-1}\left[\diag\left(\left[
                                \begin{array}{cc}
                                  W^{\left(1\right)}_{11} & W^{\left(1\right)}_{12} \\
                                  W^{\left(1\right)}_{21} & W^{\left(1\right)}_{22} \\
                                \end{array}
                              \right],\ldots,\left[
                                \begin{array}{cc}
                                  W^{\left(n_3\right)}_{11} & W^{\left(n_3\right)}_{12} \\
                                  W^{\left(n_3\right)}_{21} & W^{\left(n_3\right)}_{22} \\
                                \end{array}
                              \right]
 \right )\right]\star_M\uc^*,
\end{align*}
be partitioned accordingly to the sizes of the blocks of $\ac$.
We get
\begin{align*}
  \ac\star_M\wc=\uc\star_M\matt^{-1}\left[\diag\left(\left[
                                \begin{array}{cc}
                                  D^{\left(1\right)}W^{\left(1\right)}_{11} & D^{\left(1\right)}W^{\left(1\right)}_{12} \\
                                  O                  & O \\
                                \end{array}
                              \right],\ldots,\left[
                                \begin{array}{cc}
                                  D^{\left(n_3\right)}W^{\left(n_3\right)}_{11} &  D^{\left(n_3\right)}W^{\left(n_3\right)}_{12} \\
                                       O                   & O      \\
                                \end{array}
                              \right]
  \right)\right]\star_M\uc^*,
\end{align*}
and
\begin{align*}
  \ac\star_M\ac^\dag=\uc\star_M\matt^{-1}\left[\diag\left(\left[
                                \begin{array}{cc}
                                  I^{\left(1\right)} & O \\
                                  O       & O \\
                                \end{array}
                              \right],\ldots,\left[
                                \begin{array}{cc}
                                  I^{\left(n_3\right)} & O \\
                                       O    & O      \\
                                \end{array}
                              \right]
  \right)\right]\star_M\uc^*
\end{align*}
From $\ac\star_M\wc=\ac\star_M\ac^\dag$, we have $\wc^{\left(i\right)}_{11}=(\dc^{\left(i\right)})^{-1}$ and $\wc^{\left(i\right)}_{12}=O$, $i=\left[n_3\right]$.
In addition,
\begin{align*}
  \wc=\wc\star_M\ac\star_M\wc=\vc\star_M\matt^{-1}\left[\diag\left(\left[
                                \begin{array}{cc}
                                  (D^{\left(1\right)})^{-1}   & O \\
                                 W^{\left(1\right)}_{21}      & O \\
                                \end{array}
                              \right],\ldots,\left[
                                \begin{array}{cc}
                                  (D^{\left(n_3\right)})^{-1} & O \\
                                      W^{\left(n_3\right)}_{21} & O      \\
                                \end{array}
                              \right]
  \right)\right]\star_M\uc^*,
\end{align*}
where $\wc^{\left(i\right)}_{21}$, $i \in \left[n_3\right]$ is arbitrary. Hence from (\ref{0xxc}), we get $\wc \in \ac\left\{-,\dag\right\}$.
\\$\textbf{(b)}\Longrightarrow \textbf{(c)}$: Suppose $\wc$ is a solution of the system: $\xc\star_M\ac\star_M\xc=\xc$, $\ac\star_M\xc=\ac\star_M\ac^\dag$, then $$\ac\star_M\wc\star_M\ac=\ac\star_M\ac^\dag\star_M\ac=\ac\
\text{and}
\ \left(\ac\star_M\wc\right)^*=\left(\ac\star_M\ac^\dag\right)^*=\ac\star_M\ac^\dag=\ac\star_M\wc,$$ Therefore, $\wc \in \ac\left\{1,2,3\right\}$.\\
$\textbf{(c)}\Longrightarrow \textbf{(b)}$: Suppose $\wc \in \ac\left\{1,2,3\right\}$. Then, $\wc$ satisfies $\wc\star_M\ac\star_M\wc=\wc$.
Let $\ac$ be written in the from (\ref{333}). From $\wc \in \ac\left\{1\right\}$, by (\ref{555}) we have
\begin{align*}
  \wc=\vc\star_M\matt^{-1}\left[\diag\left(\left[
                                \begin{array}{cc}
                                  (D^{\left(1\right)})^{-1} & W^{\left(1\right)}_{12} \\
                                  W^{\left(1\right)}_{21}  & W^{\left(1\right)}_{22} \\
                                \end{array}
                              \right],\ldots,\left[
                                \begin{array}{cc}
                                  (D^{\left(n_3\right)})^{-1} & W^{\left(n_3\right)}_{12} \\
                                  W^{\left(n_3\right)}_{21} & W^{\left(n_3\right)}_{22} \\
                                \end{array}
                              \right]
  \right)\right]\star_M\uc^*,
\end{align*}
where the partition form is based on the block size of $\ac$.
Form $\left(\ac\star_M\wc\right)^*=\ac\star_M\wc$, we get $\wc^{\left(i\right)}_{12}=O$, $i\in \left[n_3\right]$. So
\begin{align*}
  \ac\star_M\wc=\uc\star_M\matt^{-1}\left[\diag\left(\left[
                                \begin{array}{cc}
                                  I^{\left(1\right)} & O\\
                                  O       & O \\
                                \end{array}
                              \right],\ldots,\left[
                                \begin{array}{cc}
                                  I^{\left(n_3\right)} & O \\
                                 O          & O \\
                                \end{array}
                              \right]
  \right)\right]\star_M\uc^*=\ac\star_M\ac^\dag.
\end{align*}
Therefore, $\wc$ satisfies $\xc\star_M\ac\star_M\xc=\xc$, $\ac\star_M\xc=\ac\star_M\ac^\dag$.
\end{proof}
\end{theorem}

\begin{theorem}
Let $\ac\in\mathbb{C}^{n_1\times n_2\times n_3}$ and  $\ac^{-,\dag}$ be a fixed 1-MP inverse. Then, the set of all 1-MP inverses of $\ac$ is given by
\begin{equation}\label{xxc}
  \ac\left\{-,\dag\right\}=\left\{\ac^{-,\dag}+\left(\ic-\ac^{-,\dag}\star_M\ac\right)\star_M\wc\star_M\ac\star_M\ac^{-,\dag}, \ \ \text{where}~\wc\in\mathbb{C}^{n_2\times n_1\times n_3}~\text{is arbitrary}\right\}.
\end{equation}
\begin{proof}
Let $\mathcal {S}:=\left\{\ac^{-,\dag}+\left(\ic-\ac^{-,\dag}\star_M\ac\right)\star_M\wc\star_M\ac\star_M\ac^{-,\dag}, \ \ \text{where}~ \wc\in\mathbb{C}^{n_2\times n_1\times n_3}~\text{is arbitrary}\right\}$, we will prove that $\ac\left\{-,\dag\right\}=\mathcal {S}$. Evidently,  $\ac^{-,\dag}+\left(\ic-\ac^{-,\dag}\star_M\ac\right)\star_M\wc\star_M\ac\star_M\ac^{-,\dag}\in\mathcal{A}\left\{1\right\}$.

By Lemma \ref{111}, $\ac^{-,\dag}\star_M\ac=\ac^{-}\star_M\ac$ and  $\ac\star_M\ac^{-,\dag}=\ac\star_M\ac^{\dag}$. Thus, we obtain
\begin{align*}
  \ac\star_M\left[\ac^{-,\dag}+\left(\ic-\ac^{-,\dag}\star_M\ac\right)\star_M\wc\star_M\ac\star_M\ac^{-,\dag}\right]
  &=\ac\star_M\ac^{-,\dag}+\ac\star_M\left(\ic-\ac^{-,\dag}\star_M\ac\right)\star_M\wc\star_M\ac\star_M\ac^{-,\dag}\\
  &=\ac\star_M\ac^{-,\dag}+\ac\star_M\left(\ic-\ac^{-}\star_M\ac\right)\star_M\wc\star_M\ac\star_M\ac^{-,\dag}\\
  &=\ac\star_M\ac^{-,\dag}=\ac\star_M\ac^\dag,
\end{align*}
which means $\ac^{-,\dag}+\left(\ic-\ac^{-,\dag}\star_M\ac\right)\star_M\wc\star_M\ac\star_M\ac^{-,\dag}\in\mathcal{A}\left\{3\right\}$.
Moreover,
\begin{align*}
 &\left[\ac^{-,\dag}+\left(\ic-\ac^{-,\dag}\star_M\ac\right)\star_M\wc\star_M\ac\star_M\ac^{-,\dag}\right]\star_M\ac\star_M\left[\ac^{-,\dag}+\left(\ic-\ac^{-,\dag}\star_M\ac\right)\star_M\wc\star_M\ac\star_M\ac^{-,\dag}\right]\\
 &=\left[\ac^{-,\dag}+\left(\ic-\ac^{-,\dag}\star_M\ac\right)\star_M\wc\star_M\ac\star_M\ac^{-,\dag}\right]\star_M\ac\star_M\ac^{-,\dag}\\
 &=\ac^{-,\dag}+\left(\ic-\ac^{-,\dag}\star_M\ac\right)\star_M\wc\star_M\ac\star_M\ac^{-,\dag}.
\end{align*}
Thus, $$\ac^{-,\dag}+\left(\ic-\ac^{-,\dag}\star_M\ac\right)\star_M\wc\star_M\ac\star_M\ac^{-,\dag}\in\mathcal{A}\left\{1,2,3\right\},$$ which implies $\mathcal {S}\subseteq \ac\{-,\dag\}$ by using Theorem \ref{113}.

On the contrary, if $\yc \in \ac\{-,\dag\}$, then there exists $\mathcal {D}\in \ac\left\{1\right\}$ such that $\yc=\mathcal {D}\star_M\mathcal {A}\star_M\mathcal {A}^\dag$. Moreover, let $\ac^{-,\dag}$ be a fixed 1-MP inverse and $\ac^{-,\dag}=\ac^-\star_M\ac\star_M\ac^\dag$, for some $\ac^- \in \ac\left\{1\right\}$. By Lemma \ref{112}, there exists $\wc\in\mathbb{C}^{n_2\times n_1\times n_3}$ such that $$\mathcal {D}=\ac^-+\wc-\ac^-\star_M\ac\star_M\wc\star_M\ac\star_M\ac^-.$$

Using Lemma \ref{111}, we have $\ac^{-,\dag}\star_M\ac=\ac^{-}\star_M\ac$ and  $\ac\star_M\ac^{-,\dag}=\ac\star_M\ac^{\dag}$. Thus, we obtain
\begin{align*}
\yc=\dc\star_M\ac\star_M\ac^\dag &=\left(\ac^-+\wc-\ac^-\star_M\ac\star_M\wc\star_M\ac\star_M\ac^-\right)\star_M\ac\star_M\ac^\dag\\
&=\ac^{-,\dag}+\left(\ic-\ac^{-,\dag}\star_M\ac\right)\star_M\wc\star_M\ac\star_M\ac^{-,\dag}.
\end{align*}
Hence, $\yc\in\mathcal {S}$, which implies $\ac\left\{-,\dag\right\}\subseteq\mathcal {S}$.
\end{proof}
\end{theorem}

\begin{theorem}
Let $\ac\in\mathbb{C}^{n_1\times n_2\times n_3}$. Then, the following statements are equivalent.
\begin{description}
  \item  [(a)] $\xc=\ac^{-,\dag}$.
  \item  [(b)] $\xc\star_M\ac\star_M\xc=\xc$, $\ac\star_M\xc\star_M\ac=\ac$, $\xc\star_M\ac=\ac^-\star_M\ac$ and $\ac\star_M\xc=\ac\star_M\ac^\dag$.
   \item  [(c)] $\xc=\xc\star_M\ac\star_M\ac^\dag$ and $\xc\star_M\ac=\ac^-\star_M\ac$.
    \item  [(d)] $\xc=\xc\star_M\ac\star_M\ac^\dag$ and $\xc\star_M\ac\star_M\ac^\dag=\ac^-\star_M\ac\star_M\ac^\dag$.
     \item  [(e)] $\xc=\xc\star_M\ac\star_M\ac^\dag$ and $\xc\star_M\ac\star_M\ac^*=\ac^-\star_M\ac\star_M\ac^*$.
\end{description}
\begin{proof}
$\textbf{(a)}\Longrightarrow\textbf{(b)}$ By Lemma \ref{111} and Theorem \ref{113}, one has this implying.\\
$\textbf{(b)}\Longrightarrow\textbf{(c)}$ Due to $\xc=\xc\star_M\left(\ac\star_M\xc\right)=\xc\star_M\ac\star_M\ac^\dag$, one has this implication. \\
$\textbf{(c)}\Longrightarrow\textbf{(d)}$ It is evident.\\
$\textbf{(d)}\Longrightarrow\textbf{(e)}$ Multiplying $\xc\star_M\ac\star_M\ac^\dag=\ac^-\star_M\ac\star_M\ac^\dag$ by $\ac\star_M\ac^*$ from the right hand side, we get
\begin{align*}
  \xc\star_M\ac\star_M\ac^\dag\star_M\ac\star_M\ac^*=\ac^-\star_M\ac\star_M\ac^\dag\star_M\ac\star_M\ac^*,
\end{align*}
which implies $\xc\star_M\ac\star_M\ac^*=\ac^-\star_M\ac\star_M\ac^*$.\\
$\textbf{(e)}\Longrightarrow$ \textbf{(a)} By $\xc=\xc\star_M\ac\star_M\ac^\dag$ and $\xc\star_M\ac\star_M\ac^*=\ac^-\star_M\ac\star_M\ac^*$, one has
\begin{align*}
 \xc&=\xc\star_M\ac\star_M\ac^\dag=\xc\star_M\ac\star_M\ac^\dag\star_M\ac\star_M\ac^\dag=\xc\star_M\ac\star_M\ac^*\star_M\left(\ac^\dag\right)^*\star_M\ac^\dag\\
 &=\ac^-\star_M\ac\star_M\ac^*\star_M\left(\ac^\dag\right)^*\star_M\ac^\dag=\ac^-\star_M\ac\star_M\ac^\dag\star_M\ac\star_M\ac^\dag\\
 &=\ac^-\star_M\ac\star_M\ac^\dag.
\end{align*}
\end{proof}
\end{theorem}

\subsection{The 1-D inverse of a tensor}

In this subsection, we will study the 1-D inverse of the tensor under the M-product.

\begin{definition}\cite{SPBS}
Let $\ac \in \CC^{n_1\times n_1\times n_3}$ and $\ind\left(\ac\right) = k$. The unique solution $\wc$ to the following equations
\begin{equation}\label{mpk2}
\textnormal{(I)} \ \wc\star_M\ac^{k+1} = \ac^k, \quad \textnormal{(II)} \ \wc\star_M\ac\star_M\wc = \wc, \quad \textnormal{(III)} \ \ac\star_M\wc = \wc\star_M\ac,
\end{equation}
is called the \textbf{Drazin inverse} of the tensor $\ac$ and is denoted by $\ac^D$.
\end{definition}

\begin{lemma}\label{ffd}
Let $\ac\in\mathbb{C}^{n_1\times n_1\times n_3}$ with $\ind\left(\ac\right) = k$ and $\ac^-$ be a fixed $\left\{1\right\}$-inverse of $\ac$. Then, $\ac^-\star_M\ac\star_M\ac^D \in \mathbb{C}^{n_1\times n_1\times n_3}$ is the unique solution of the following system of tensor equations:
\begin{align}\label{abc}
 \xc\star_M\ac\star_M\xc = \xc,~\xc\star_M\ac^k=\ac^-\star_M\ac^k~\text{and}~\ac\star_M\xc=\ac\star_M\ac^D.
\end{align}
\begin{proof}
Let $\xc=\ac^-\star_M\ac\star_M\ac^D$. Then, we have
\begin{align*}
 \xc\star_M\ac\star_M\xc&=\ac^-\star_M\ac\star_M\ac^D\star_M\ac\star_M\ac^-\star_M\ac\star_M\ac^D=\ac^-\star_M\ac\star_M\ac^D\star_M\ac\star_M\ac^D\\
 &=\ac^-\star_M\ac\star_M\ac^D=\xc,
\end{align*}
\begin{align*}
  \xc\star_M\ac^k=\ac^-\star_M\ac\star_M\ac^D\star_M\ac^k=\ac^-\star_M\ac^D\star_M\ac^{k+1}=\ac^-\star_M\ac^k,
\end{align*}
and $\ac\star_M\xc=\ac\star_M\ac^-\star_M\ac\star_M\ac^D=\ac\star_M\ac^D$. Hence, $\xc$ is a solution of (\ref{abc}).

In the following, we will prove the uniqueness of $\xc$. Now, we suppose that there is another solution $\yc$, satisfying the equation
(\ref{abc}). We will have
\begin{align*}
  \yc=\yc\star_M\ac\star_M\yc=\yc\star_M\ac\star_M\ac^D=\yc\star_M\ac^k\star_M\left(\ac^D\right)^k=\ac^-\star_M\ac^k\star_M\left(\ac^D\right)^k=\ac^-\star_M\ac\star_M\ac^D=\xc.
\end{align*}
\end{proof}
\end{lemma}
By Lemma \ref{ffd}, we can  give the definition  of the 1-D inverse for
tensors.
\begin{definition}
Let $\ac\in\mathbb{C}^{n_1\times n_1\times n_3}$ with $\ind\left(\ac\right) = k$ and $\ac^-$ be a fixed $\left\{1\right\}$-inverse of $\ac$. A tensor $\xc\in\mathbb{C}^{n_1\times n_1\times n_3}$ is called the 1-D inverse of $\ac$ if it satisfies

\begin{equation}\label{mpkk12}
\textnormal{(I)} \ \xc\star_M\ac\star_M\xc = \xc, \quad \textnormal{(II)} \ \xc\star_M\ac^k=\ac^-\star_M\ac^k, \quad \textnormal{(III)} \ \ac\star_M\xc=\ac\star_M\ac^D.
\end{equation}
We denote the 1-D inverse of $\ac$ by $\ac^{-,\ D}$. Clearly, $\ac^{-,\ D}=\ac^-\star_M\ac\star_M\ac^D$.

\end{definition}


In the following, we will give an algorithm to compute the 1-D inverse.

\begin{algorithm}[H]
\caption{Computing the 1-D inverse under the M-product}	
\KwIn{$\ac\in\mathbb{C}^{n_1\times n_1\times n_3}$ and $M \in\mathbb{C}^{n_3\times n_3}$}
\KwOut {$\xc=\ac^{-,\ D}$}
\begin{enumerate}\label{a2}
\item Compute $\widehat{\ac}=\ac\times_3M$
\item $k=\ind\left(\widehat{\ac}\right)$
\item $\textbf{For}$ $i=1:n_3$ \textbf{do}
\item ~~~~~~$\widehat{\ac^D}^{\left(i\right)}=\widehat{\ac^k}^{\left(i\right)}\cdot\left(\left(\widehat{\ac^{2k+1}}\right)^\dag\right)^{\left(i\right)}\cdot\widehat{\ac^k}^{\left(i\right)}$
\item \textbf{End for}
\item $\textbf{For}$ $i=1:n_3$ \textbf{do}
\item ~~~~~~$\widehat{\pc}^{\left(i\right)}=\widehat{\ac^{-}}^{\left(i\right)}\cdot\widehat{\ac}^{\left(i\right)}\cdot\widehat{\ac^D}^{\left(i\right)}$
\item \textbf{End for}
\item Compute $\xc=\widehat{\pc}\times_3M^{-1}$
\item \textbf{Return} $\ac^{-,\ D}=\xc$.
\end{enumerate}
\end{algorithm}

\begin{example}
Let $\ac\in\mathbb{C}^{3\times3\times3}$ and $M\in\mathbb{C}^{3\times3}$  with entries
$$\ac^{\left(1\right)}=
\begin{bmatrix}
0 & 1   & 1\\
0 & -1   & 0\\
1 & -1  & 1
\end{bmatrix},~
\ac^{\left(2\right)}=
\begin{bmatrix}
1 & 0 & 1 \\
0 & 0 & 0 \\
1 & 0 & 0
\end{bmatrix},~
\ac^{(3)}=
\begin{bmatrix}
0   & 0& 0\\
0   & 1& 0\\
-1  & 1& 0
\end{bmatrix},~
M=
\begin{bmatrix}
1 & 0 & 1\\
0 & 1 & 0 \\
0 & 1 & 1
\end{bmatrix}.
$$
Since $\ind\left(\widehat{\ac}^{\left(1\right)}\right)=\ind\left(\widehat{\ac}^{\left(2\right)}\right)=\ind\left(\widehat{\ac}^{\left(3\right)}\right)=2$, we get  the  index of $\ac$ is $k=2$.  Now, if we fix one $\left\{1\right\}$-inverse of $\ac$, we can get
$$
\left(\ac^{-}\right)^{\left(1\right)}=
\begin{bmatrix}
1 & 1 & 1 \\
2 & 1 & -1 \\
0 & 2 & 1
\end{bmatrix},\
\left(\ac^{-}\right)^{\left(2\right)}=
\begin{bmatrix}
0 & 1 & 1 \\
1 & 1 & 1 \\
1 & 1 & -1
\end{bmatrix},\
\left(\ac^{-}\right)^{\left(3\right)}=
\begin{bmatrix}
0 & 0 & 0 \\
-1 & -1  & 0 \\
0       & -2      & 0
\end{bmatrix}.
$$
By Algorithm \ref{a2}, we can calculate $\xc=\ac^{-,\ D}$, that is
$$\mathcal X^{\left(1\right)}=
\begin{bmatrix}
0 & -2 & 3\\
1 & -1 & 1\\
0 & 3  & -1
\end{bmatrix},\
\mathcal X^{\left(2\right)}=
\begin{bmatrix}
0 & 0 & 1 \\
1 & 0 & 1 \\
1 & 0 & -1
\end{bmatrix},\
\mathcal X^{\left(3\right)}=
\begin{bmatrix}
0 & 2 & -1 \\
-1& 1 &  -1\\
0 &-3 & 2
\end{bmatrix}.
$$
\end{example}

Next, we discuss several characteristics of the 1-D inverse of a tensor.

\begin{lemma}\label{222}
Let $\gc\in\mathbb{C}^{n_1\times n_2\times n_3}$ and $\hc\in\mathbb{C}^{n_2\times n_1\times n_3}$. Then
$$\left(\gc\star_M\hc\right)^D=\gc\star_M\left[\left(\hc\star_M\gc\right)^D\right]^2\star_M\hc$$.
\begin{proof}
By (\ref{mmm}), one has
\begin{eqnarray*}
  \gc=\matt^{-1}\left[\matt(\gc)\right]=\matt^{-1}\left[\diag\left(
                           \widehat{\gc}^{\left(1\right)},\ldots,
                           \widehat{\gc}^{\left(n_3\right)}
                         \right)\right],
\end{eqnarray*}
and
\begin{eqnarray*}
  \hc=\matt^{-1}\left[\matt\left(\hc\right)\right]=\matt^{-1}\left[\diag\left(
                           \widehat{\hc}^{\left(1\right)},\ldots,
                           \widehat{\hc}^{\left(n_3\right)}
                         \right)\right].
\end{eqnarray*}
By {\rm \cite{A}}, We get $\left(\widehat{\gc}^{\left(i\right)}\widehat{\hc}^{\left(i\right)}\right)^D
=\widehat{\gc}^{\left(i\right)}\left[\left(\widehat{\hc}^{\left(i\right)}\widehat{\gc}^{\left(i\right)}\right)^D\right]^2\widehat{\hc}^{\left(i\right)}$, $i=[n_3]$.
Then
\begin{align*}
\left(\gc\star_M\hc\right)^D&=\matt^{-1}\left[\diag\left(\widehat{\gc}^{\left(1\right)}\left[\left(\widehat{\hc}^{\left(1\right)}\widehat{\gc}^{\left(1\right)}\right)^D\right]^2\widehat{\hc}^{\left(1\right)}
,\ldots,\widehat{\gc}^{\left(n_3\right)}\left[\left(\widehat{\hc}^{\left(n_3\right)}\widehat{\gc}^{\left(n_3\right)}\right)^D\right]^2\widehat{\hc}^{\left(n_3\right)}\right)\right]\\
&=\gc\star_M\left[\left(\hc\star_M\gc\right)^D\right]^2\star_M\hc.
\end{align*}
\end{proof}
\end{lemma}

\begin{theorem}
Let $\ac\in\mathbb{C}^{n_1\times n_1\times n_3}$ with $\ind\left(\ac\right) = k$ and $\ac^-$ be a fixed $\left\{1\right\}$-inverse of $\ac$. Then $\ac^{-,\ D}=\left(\ac^-\star_M\ac^2\right)^D$.
\begin{proof}
By Lemma \ref{222}, we have
\begin{align*}
 \left(\ac^-\star_M\ac^2\right)^D&=\left(\ac^-\star_M\ac\star_M\ac\right)^D=\ac^-\star_M\ac\star_M\left[\left(\ac\star_M\ac^-\star_M\ac\right)^D\right]^2\star_M\ac\\
 &=\ac^-\star_M\ac\star_M\left(\ac^D\right)^2\star_M\ac=\ac^-\star_M\ac\star_M\ac^D\star_M\ac\star_M\ac^D\\
 &=\ac^-\star_M\ac\star_M\ac^D=\ac^{-,\ D}.
\end{align*}
\end{proof}
\end{theorem}

\begin{theorem}
Let $\ac\in\mathbb{C}^{n_1\times n_1\times n_3}$ with $\ind\left(\ac\right) = k$ and $\ac^-$ be a fixed $\left\{1\right\}$-inverse of $\ac$. Then $\ac^{-,\ D}=\ac^D$ if and only if $\ac^D\star_M\ac^k=\ac^-\star_M\ac^k$.

\begin{proof}
$(\Longrightarrow)$:   $\ac^D\star_M\ac^k=\ac^{-,\ D}\star_M\ac^k=\ac^-\star_M\ac\star_M\ac^D\star_M\ac^k=\ac^-\star_M\ac^k$.

$(\Longleftarrow)$: We know that $\ac\star_M\ac^D=\ac^k\star_M\left(\ac^D\right)^k$ from the properties of the Drazin inverse. Thus, we have
\begin{align*}
  \ac^{-,\ D}=\ac^-\star_M\ac\star_M\ac^D=\ac^-\star_M\ac^k\star_M\left(\ac^D\right)^k=\ac^D\star_M\ac^k\star_M\left(\ac^D\right)^k=\ac^D\star_M\ac\star_M\ac^D=\ac^D.
\end{align*}
\end{proof}
\end{theorem}

Now, we will discuss the idempotent property of the 1-D inverse.

\begin{lemma} \label{aaa}
Let $\ac\in\mathbb{C}^{n_1\times n_1\times n_3}$ with $\ind\left(\ac\right) = k$ and $\ac^-$ be a fixed $\left\{1\right\}$-inverse of $\ac$. Then
\begin{description}
  \item  [(a)] $\left(\ac^{-,\ D}\right)^2=\ac^-\star_M\ac^D$.
  \item  [(b)] $\ac^{-,\ D}$ is idempotent if and only if $\ac^{-,\ D}=\ac^-\star_M\ac^D$ if and only if $\ac^{-,\ D}=\ac^{-,\ D}\star_M\ac$.
  \end{description}
  \begin{proof} \textbf{(a)} A simple computation gives
    \begin{align*}
    \left(\ac^{-,\ D}\right)^2&=\ac^{-,\ D}\star_M\ac^{-,\ D}=\ac^-\star_M\ac\star_M\ac^D\star_M\ac^-\star_M\ac\star_M\ac^D\\
        &=\ac^-\star_M\ac^D\star_M\left(\ac\star_M\ac^-\star_M\ac\right)\star_M\ac^D=\ac^-\star_M\left(\ac^D\star_M\ac\star_M\ac^D\right)\\
        &=\ac^-\star_M\ac^D.
    \end{align*}
  \textbf{(b)} The first equivalence is followed by (a). Now, let us prove that $\ac^{-,\ D}$ is idempotent if and
only if $\ac^{-,\ D}=\ac^{-,\ D}\star_M\ac$. In fact, if $\ac^{-,\ D}$ is idempotent, by (a), we have $\ac^{-,\ D}=\ac^-\star_M\ac^D$.
Now, $$\ac^{-,\ D}=\ac^-\star_M\ac\star_M\ac^D=\ac^-\star_M\ac^D\star_M\ac=\ac^{-,\ D}\star_M\ac.$$
On the contrary,
\begin{align*}
 \ac^{-,\ D}&=\ac^{-,\ D}\star_M\ac=\ac^-\star_M\ac\star_M\ac^D\star_M\ac=\ac^-\star_M\ac\star_M\ac^D\star_M\ac\star_M\ac^D\star_M\ac\\
 &=\ac^-\star_M\ac\star_M\ac^D\star_M\ac\star_M\ac^-\star_M\ac\star_M\ac^D\star_M\ac=\ac^{-,\ D}\star_M\ac\star_M\ac^{-,\ D}\star_M\ac\\
 &=\ac^{-,\ D}\star_M\ac^{-,\ D}=\left(\ac^{-,\ D}\right)^2.
\end{align*}
 \end{proof}
\end{lemma}

\begin{theorem}\label{aas}
Let $\ac\in\mathbb{C}^{n_1\times n_1\times n_3}$ with $\ind\left(\ac\right) = k$ and $\ac^-$ be a fixed $\left\{1\right\}$-inverse of $\ac$. If $\ac^{-,\ D}$ is idempotent, then:
\begin{description}
  \item  [(a)] $\ac^k=\ac^{k+1}$. In addition, $\ac^D\star_M\ac^k=\ac^k$.
  \item  [(b)] $\left(\ac^{-,\ D}\right)^k=\left(\ac^{-,\ D}\right)^k\star_M\ac$. In addition,  $\left(\ac^{-,\ D}\right)^m=\left(\ac^{-,\ D}\right)^m\star_M\ac$ for every $m \in \mathbb{Z}^+$.
  \item  [(c)] $\ac^{-,\ D}=\left(\ac^{-,\ D}\right)^m\star_M\ac^m$ for every $m \in \mathbb{Z}^+$.
  \item  [(d)] $\ac^k\star_M\ac^{-,\ D}=\ac^k$.
  \end{description}
 \begin{proof} \textbf{(a)} By Lemma \ref{aaa} \textbf{(b)}, one has
 \begin{align*}
 \ac^k&=\ac^{k+1}\star_M\ac^D=\ac^k\star_M\ac\star_M\ac^-\star_M\ac\star_M\ac^D\\
 &=\ac^k\star_M\ac\star_M\ac^-\star_M\ac^D\star_M\ac=\ac^k\star_M\ac\star_M\ac^{-,\ D}\star_M\ac\\
 &=\ac^k\star_M\ac\star_M\ac^-\star_M\ac\star_M\ac^D\star_M\ac=\ac^k\star_M\ac\star_M\ac^D\star_M\ac\\
 &=\ac^k\star_M\ac=\ac^{k+1}.
 \end{align*}
In addition, $\ac^k=\ac^{k+1}\star_M\ac^D=\ac^{k}\star_M\ac^D=\ac^D\star_M\ac^{k}$.

\textbf{(b)} By Lemma \ref{aaa} \textbf{(b)}, we get $\left(\ac^{-,\ D}\right)^k=\ac^{-,\ D}=\ac^{-,\ D}\star_M\ac=\left(\ac^{-,\ D}\right)^k\star_M\ac$.

\textbf{(c)} By using Lemma \ref{aaa} \textbf{(b)} again, one has
  \begin{align*}
   \ac^{-,\ D}&=\left(\ac^{-,\ D}\right)^m=\left(\ac^{-,\ D}\right)^{m-1}\star_M\ac^{-,\ D}=\left(\ac^{-,\ D}\right)^{m-1}\star_M\ac^{-,\ D}\star_M\ac=\left(\ac^{-,\ D}\right)^{m-1}\star_M\ac^{-,\ D}\star_M\ac^2\\
   &=\left(\ac^{-,\ D}\right)^{m-1}\star_M\ac^{-,\ D}\star_M\ac^3=\ldots=\left(\ac^{-,\ D}\right)^m\star_M\ac^m.
  \end{align*}

\textbf{(d)} By using \textbf{(a)}, one has
  \begin{align*}
    \ac^k\star_M\ac^{-,\ D}=\ac^{k+1}\star_M\ac^-\star_M\ac\star_M\ac^D=\ac^{k+1}\star_M\ac^D
      =\ac^k.
  \end{align*}
  \end{proof}
\end{theorem}

In the following, we characterize the 1-D inverse by using tensor equations.
\begin{theorem}
Let $\ac\in\mathbb{C}^{n_1\times n_1\times n_3}$ with $\ind\left(\ac\right) = k$ and $\ac^-$ be a fixed $\left\{1\right\}$-inverse of $\ac$. Then the following statements are equivalent:
\begin{description}
  \item  [(a)] $\ac^{-,\ D}=\xc$.
  \item  [(b)] $\xc\star_M\ac\star_M\xc=\xc$, $\xc\star_M\ac^k=\ac^-\star_M\ac^k$, $\ac\star_M\xc\star_M\ac=\ac\star_M\ac^D\star_M\ac$ and $\ac\star_M\xc=\ac\star_M\ac^D$.
  \item  [(c)] $\ac^-\star_M\ac\star_M\xc=\xc$, $\xc\star_M\ac^k=\ac^-\star_M\ac^k$ and $\xc=\xc\star_M\ac\star_M\ac^D$.
  \item  [(d)] $\ac^-\star_M\ac\star_M\xc\star_M\ac\star_M\ac^D=\xc$ and $\ac\star_M\xc\star_M\ac^k=\ac^k$.
  \item  [(e)] $\ac^-\star_M\ac\star_M\ac^D\star_M\ac=\xc\star_M\ac$, $\ac^k\star_M\xc=\ac^k\star_M\ac^D$ and $\xc=\xc\star_M\ac\star_M\ac^D$.
  \end{description}
  \begin{proof}
\textbf{(a)}$\Longrightarrow$ \textbf{(b)} Let $\xc=\ac^{-,\ D}$. By Lemma \ref{ffd}, it is enough to see $\ac\star_M\xc\star_M\ac=\ac\star_M\ac^D\star_M\ac$.  Moreover,  $$\ac\star_M\xc\star_M\ac=\ac\star_M\ac^{-,D}\star_M\ac=\ac\star_M\ac^-\star_M\ac\star_M\ac^D\star_M\ac=\ac\star_M\ac^D\star_M\ac.$$
\textbf{(b)} $\Longrightarrow$ \textbf{(c)} By $\xc\star_M\ac\star_M\ac^D=\xc\star_M\ac\star_M\xc=\xc$ and $\xc\star_M\ac^k=\ac^-\star_M\ac^k$, one has
\begin{align*}
  \ac^-\star_M\ac\star_M\xc&=\ac^-\star_M\ac\star_M\ac^D=\ac^-\star_M\ac^k\star_M\left(\ac^D\right)^k=\xc\star_M\ac^k\star_M\left(\ac^D\right)^k=\xc\star_M\ac\star_M\ac^D=\xc.
\end{align*}
\textbf{(c)} $\Longrightarrow$ \textbf{(d)} The implication is true since $\ac^-\star_M\ac\star_M\xc=\xc$ implies $\ac\star_M\xc\star_M\ac=\ac^k$ trivially and $$\ac\star_M\xc\star_M\ac^k=\ac\star_M\ac^-\star_M\ac^k=\ac\star_M\ac^-\star_M\ac\star_M\ac^{k-1}=\ac^k$$ trivially and $\ac^-\star_M\ac\star_M\xc\star_M\ac\star_M\ac^D=\xc\star_M\ac\star_M\ac^D=\xc$.\\
\textbf{(d)} $\Longrightarrow$ \textbf{(a)} Suppose $\ac^-\star_M\ac\star_M\xc\star_M\ac\star_M\ac^D=\xc$ and $\ac\star_M\xc\star_M\ac^k=\ac^k$, one has
\begin{align*}
 \xc&=\ac^-\star_M\ac\star_M\xc\star_M\ac\star_M\ac^D=\ac^-\star_M\ac\star_M\xc\star_M\ac^k\star_M\left(\ac^D\right)^k\\
 &=\ac^-\star_M\ac^k\star_M\left(\ac^D\right)^k=\ac^-\star_M\ac\star_M\ac^D=\ac^{-,\ D}.
\end{align*}
(a) $\Longrightarrow$ (e) It is simple to see
$$\xc\star_M\ac\star_M\ac^D=\ac^-\star_M\ac\star_M\ac^D\star_M\ac\star_M\ac^D=\ac^-\star_M\ac\star_M\ac^D=\xc$$
and
$$\ac^k\star_M\xc=\ac^k\star_M\ac^{-,\ D}=\ac^k\star_M\ac^-\star_M\ac\star_M\ac^D=\ac^k\star_M\ac^D.$$
(e) $\Longrightarrow$ (a) It follows from $\xc=\xc\star_M\ac\star_M\ac^D=\ac^-\star_M\ac\star_M\ac^D\star_M\ac\star_M\ac^D=\ac^-\star_M\ac\star_M\ac^D=\ac^{-,\ D}$.\\
  \end{proof}
\end{theorem}

\begin{theorem}
Let $\ac\in\mathbb{C}^{n_1\times n_1\times n_3}$ with $\ind\left(\ac\right) = k$ and $\ac^-$ be a fixed $\left\{1\right\}$-inverse of $\ac$. Then
\begin{description}
  \item  [(a)] $\ac^{-,\ D}\star_M\ac=\ac^-\star_M\ac$ if and only if $\ac\star_M\ac^D\star_M\ac=\ac$.
  \item  [(b)] $\ac^{-,\ D}\star_M\ac=\ac^D\star_M\ac$ if and only if $\ac^{-,\ D}=\ac^D$.
  \item  [(c)] $\ac^k\star_M\ac^-\star_M\ac^k=\ac^k$ if and only if $\ac^k\star_M\ac^{-,\ D}\star_M\ac^k=\ac^k$.
  \item  [(d)] $\ac^{-,\ D}=\ac^D$ if and only if $\ac^{-,\ D}\star_M\ac=\ac\star_M\ac^{-,\ D}$.
  \end{description}
  \begin{proof}
  \textbf{(a)} Let $\ac^{-,\ D}\star_M\ac=\ac^-\star_M\ac$. Then,$$\ac=\ac\star_M\ac^-\star_M\ac=\ac\star_M\ac^{-,\ D}\star_M\ac=\ac\star_M\ac^-\star_M\ac\star_M\ac^D\star_M\ac=\ac\star_M\ac^D\star_M\ac. $$
   On the contrary, it is also  true.\\
\textbf{(b)} It is can be derived by the definition.\\
\textbf{(c)} Let $\ac^k\star_M\ac^-\star_M\ac^k=\ac^k$. Then
\begin{align*}
 \ac^k=\ac^k\star_M\ac^-\star_M\ac^k=\ac^k\star_M\ac^-\star_M\ac^D\star_M\ac^{k+1}
 =\ac^k\star_M\ac^-\star_M\ac\star_M\ac^D\star_M\ac^k=\ac^k\star_M\ac^{-,\ D}\star_M\ac^k.
\end{align*}
Conversely,
\begin{align*}
  \ac^k\star_M\ac^-\star_M\ac^k&=\ac^k\star_M\ac^-\star_M\ac^D\star_M\ac^{k+1}=\ac^k\star_M\ac^-\star_M\ac\star_M\ac^D\star_M\ac^k\\
  &=\ac^k\star_M\ac^{-,\ D}\star_M\ac^k=\ac^k.
\end{align*}
\textbf{(d) } Let $\ac^{-,\ D}=\ac^D$. Then,
$$\ac^{-,\ D}\star_M\ac=\ac^D\star_M\ac=\ac\star_M\ac^D=\ac\star_M\ac^-\star_M\ac\star_M\ac^D=\ac\star_M\ac^{-,\ D}.$$
Otherwise, suppose $\ac^{-,\ D}\star_M\ac=\ac\star_M\ac^{-,\ D}$. By the definition, one has $\ac^{-,\ D}\star_M\ac\star_M\ac^{-,\ D}=\ac^{-,\ D}$. On the other hand,
\begin{align*}
  \ac^{k+1}\star_M\ac^{-,\ D}=\ac^k\star_M\ac\star_M\ac^-\star_M\ac\star_M\ac^D=\ac^{k+1}\star_M\ac^D=\ac^k.
\end{align*}
Therefore, by uniqueness of Drazin inverse, we have $\ac^D=\ac^{-,\ D}$.
 \end{proof}
\end{theorem}

\begin{theorem}
Let $\ac\in\mathbb{C}^{n_1\times n_1\times n_3}$ with $\ind\left(\ac\right) = k$ and $\ac^-$ be a fixed $\left\{1\right\}$-inverse of $\ac$ and let $m\in\mathbb{Z}^+ \backslash \left\{1\right\}$. Then
$${\left(\ac^{-,\ D}\right)}^m=\left\{\begin{array}{l}{\left(\ac^-\star_M\ac^D\right)}^\frac m2,~~\text{if~m~is~even}.\\
\\\ac\star_M{\left(\ac^D\right)}^\frac{m+1}2,~~\text{if~m~is~odd}.\end{array}\right.
$$
\begin{proof}
Suppose $m$ is even, then $m=2k$ for some positive integer $k$. Now, by Lemma \ref{aaa}, one has
 $$\left(\ac^{-,\ D}\right)^m=\left(\left(\ac^{-,\ D}\right)^2\right)^k=\left(\ac^-\star_M\ac^D\right)^k=\left(\ac^-\star_M\ac^D\right)^\frac m2.$$
Suppose $m$ is odd, then $m=2l+1$ for some positive integer $l$. Again, using Lemma \ref{aaa}, one has
 \begin{align*}
   \left(\ac^{-,\ D}\right)^m&=\left(\left(\ac^{-,\ D}\right)^2\right)^l\star_M\ac^{-,\ D}=\left(\ac^-\star_M\ac^D\right)^l\star_M\ac^-\star_M\ac\star_M\ac^D\\
   &=\left(\ac^-\star_M\ac^D\right)^{l-1}\star_M\ac^-\star_M\ac^D\star_M\ac^-\star_M\ac\star_M\ac^D\\
   &=\left(\ac^-\star_M\ac^D\right)^{l-1}\star_M\ac^-\star_M\ac^D\star_M\ac\star_M\ac^D\star_M\ac^-\star_M\ac\star_M\ac^D\\
   &=\left(\ac^-\star_M\ac^D\right)^{l-1}\star_M\ac^-\star_M\ac^D\star_M\ac^D\star_M\ac\star_M\ac^-\star_M\ac\star_M\ac^D\\
   &=\left(\ac^-\star_M\ac^D\right)^{l-1}\star_M\ac^-\star_M\ac^D\star_M\ac^D\star_M\ac\star_M\ac^D\\
   &=\left(\ac^-\star_M\ac^D\right)^{l-1}\star_M\ac^-\star_M\left(\ac^D\right)^2\\
   &=\ldots=\ac\star_M\left(\ac^D\right)^{l+1}=\ac\star_M\left(\ac^D\right)^\frac {m+1}2.
 \end{align*}
\end{proof}
\end{theorem}

\subsection{The 1-Star inverse of a tensor}

In this subsection, we will study the 1-Star inverse of a tensor.

\begin{definition}\label{d03}
Let $\ac \in\mathbb{C}^{n_1\times n_2\times n_3}$. For each $\ac^-\in
\ac\left\{1\right\}$, the tensor
\begin{equation}\label{mp013}
\ac^{-,*}=\ac^-\star_M\ac\star_M\ac^*
\end{equation}
is called the \textbf{1-Star inverse} of the tensor $\ac$ and is denoted by $\ac^{-,*}$.
The symbol $\ac\left\{-,*\right\}$ stands for the set of all 1-Star inverses of $\ac$.
\end{definition}

Notice that the 1-Star inverses is not unique. We use $\ac\left\{-,*\right\}=\left\{\ac^{-,*}|\ac \in\mathbb{C}^{n_1\times n_2\times n_3}\right\}$ to denote the set of the 1-Star inverses of a tensor. In the following, an algorithm for computing the 1-Star inverse is established.

\begin{algorithm}[H]
\caption{Computing the 1-Star inverse under the M-product}	
\KwIn{$\ac\in\mathbb{C}^{n_1\times n_2\times n_3}$ and $M \in\mathbb{C}^{n_3\times n_3}$}
\KwOut {$\xc=\ac^{-,*}$}
\begin{enumerate} \label{a3}
\item Compute $\widehat{\ac}=\ac\times_3M$
\item $\textbf{For}$ $i=1:n_3$ \textbf{do}
\item ~~~~~~$\widehat{\pc}^{\left(i\right)}=\widehat{\ac^{-}}^{\left(i\right)}\cdot\widehat{\ac}^{\left(i\right)}\cdot\widehat{\ac^*}^{\left(i\right)}$
\item \textbf{End for}
\item Compute $\xc=\widehat{\pc}\times_3M^{-1}$
\item \textbf{Return} $\ac^{-,*}=\xc$.
\end{enumerate}
\end{algorithm}

\begin{example}
Let $\ac\in\mathbb{C}^{2\times3\times3}$ and $M\in\mathbb{C}^{3\times3}$  with entries
$$(\ac)^{\left(1\right)}=
\begin{bmatrix}
2 & 0 & 0\\
0 & 2 & 1
\end{bmatrix},~
(\ac)^{(2)}=
\begin{bmatrix}
1 & 0 & 1 \\
1 & 1 & 1
\end{bmatrix},~
(\ac)^{(3)}=
\begin{bmatrix}
-1 & 0 & 0 \\
0  &-1 & -1
\end{bmatrix},~
M=
\begin{bmatrix}
1 & 0 & 1\\
0 & 1 & 0 \\
0 & 1 & 1
\end{bmatrix}.
$$
We can get
$$
(\widehat{\ac^{-}})^{\left(1\right)}=
\begin{bmatrix}
1      & 0  \\
0      & 1  \\
a_{31} & a_{32}
\end{bmatrix},\
(\widehat{\ac^{-}})^{\left(2\right)}=
\begin{bmatrix}
b_{11}   & b_{12} \\
1-b_{11} & 1-b_{12} \\
4                 & 5
\end{bmatrix},\
(\widehat{\ac^{-}})^{\left(3\right)}=
\begin{bmatrix}
0      & 1  \\
c_{21} & c_{22} \\
1      & 0
\end{bmatrix},
$$
By Algorithm \ref{a3}, we can calculate $\xc=\ac^{-,*}$ as
$$\mathcal X^{\left(1\right)}=
\begin{bmatrix}
2b_{11}+2b_{12}+1        & 2b_{11}+3b_{12}-1 \\
4-2b_{12}-c_{21}-2b_{11} & 6-3b_{12}-c_{22}-2b_{11} \\
a_{31}+3 & a_{32}+5
\end{bmatrix},$$
$$
\mathcal X^{\left(2\right)}=
\begin{bmatrix}
2b_{11}+2b_{12}   & 2b_{11}+2b_{12}  \\
4-2b_{12}-2b_{11} & 5-3b_{12}-3b_{11}  \\
4                 & 5
\end{bmatrix},~
\mathcal X^{\left(3\right)}=
\begin{bmatrix}
-2b_{11}-2b_{12}         & 1-3b_{12}-2b_{11}  \\
2b_{11}+2b_{12}+c_{21}-4 & 2b_{11}+3b_{12}+c_{22}-5 \\
-3                       &  -5
\end{bmatrix}.
$$
\end{example}

The following theorem shows that the 1-Star inverse is the unique solution to the system (\ref{mpkk13}).

\begin{theorem}
Let $\ac\in \mathbb{C}^{n_1\times n_2\times n_3} $ and $\ac^-\in \ac\left\{1\right\}$. Then $\xc=\ac^{-,*}=\ac^-\star_M\ac\star_M\ac^*$ is the unique solution to the system of equations
\begin{equation}\label{mpkk13}
\textnormal{(I)} \ \xc\star_M\left(\ac^\dag\right)^*\star_M\xc = \xc, \quad \textnormal{(II)} \ \ac\star_M\xc=\ac\star_M\ac^*, \quad \textnormal{(III)} \ \xc\star_M\left(\ac^\dag\right)^*=\ac^-\star_M\ac.
\end{equation}
\begin{proof}
Firstly, let us prove the existence. Suppose $\xc=\ac^-\star_M\ac\star_M\ac^*$, then
\begin{align*}
 \textnormal{(I)}~~  \xc\star_M\left(\ac^\dag\right)^*\star_M\xc&=\ac^-\star_M\ac\star_M\ac^*\star_M\left(\ac^\dag\right)^*\star_M\ac^-\star_M\ac\star_M\ac^*\\
 &=\ac^-\star_M\ac\star_M\ac^\dag\star_M\ac\star_M\ac^-\star_M\ac\star_M\ac^*\\
 &=\ac^-\star_M\ac\star_M\ac^-\star_M\ac\star_M\ac^*\\
 &=\ac^-\star_M\ac\star_M\ac^*=\xc,
\end{align*}
\begin{align*}
\textnormal{(II)}~~ \ac\star_M\xc=\ac\star_M\ac^-\star_M\ac\star_M\ac^*=\ac\star_M\ac^*,
\end{align*}
\begin{align*}
\textnormal{(III)}~~\xc\star_M\left(\ac^\dag\right)^*=\ac^-\star_M\ac\star_M\ac^*\star_M\left(\ac^\dag\right)^*=\ac^-\star_M\ac\star_M\ac^\dag\star_M\ac=\ac^-\star_M\ac.
\end{align*}
Next, let's prove the uniqueness. Let $\xc$ and $\yc$ be two solutions to the tensor equation (\ref{mpkk13}), then
$$\xc=\xc\star_M\left(\ac^\dag\right)^*\star_M\xc=\ac^-\star_M\ac\star_M\xc=\ac^-\star_M\ac\star_M\ac^*=\ac^-\star_M\ac\star_M\yc=\yc\star_M\left(\ac^\dag\right)^*\star_M\yc=\yc.$$
Therefore, $\xc=\ac^{-,*}=\ac^-\star_M\ac\star_M\ac^*$ is the unique solution to the system   (\ref{mpkk13}).
\end{proof}
\end{theorem}

Next, we discuss several characteristics of the 1-Star inverse of a tensor.

\begin{theorem}
Let $\ac\in\mathbb{C}^{n_1\times n_2\times n_3}$ and $\ac^-\in \ac\left\{1\right\}$. Then, the following statements are equivalent.
\begin{description}
  \item  [(a)] $\xc=\ac^-\star_M\ac\star_M\ac^*$,
  \item  [(b)] $\xc\star_M\left(\ac^\dag\right)^*\star_M\xc=\xc$, $\left(\ac^\dag\right)^*\star_M\xc\star_M\left(\ac^\dag\right)^*=\left(\ac^\dag\right)^*$, $\ac\star_M\xc=\ac\star_M\ac^*$, $\xc\star_M(\ac^\dag)^*=\ac^-\star_M\ac$,
  \item  [(c)] $\xc\star_M\left(\ac^\dag\right)^*\star_M\xc=\xc$,
      $\left(\ac^\dag\right)^*\star_M\xc\star_M\left(\ac^\dag\right)^*=\left(\ac^\dag\right)^*$, $(\ac^\dag)^*\star_M\xc=\ac\star_M\ac^\dag$, $\xc\star_M(\ac^\dag)^*=\ac^-\star_M\ac$,
  \item  [(d)] $\xc\star_M\left(\ac^\dag\right)^*\star_M\xc=\xc$, $\left(\ac^\dag\right)^*\star_M\xc=\ac\star_M\ac^\dag$, $\xc\star_M\left(\ac^\dag\right)^*=\ac^-\star_M\ac$,
  \item  [(e)] $\ac^-\star_M\ac\star_M\xc=\xc$, $\ac\star_M\xc=\ac\star_M\ac^*$,
\end{description}
\end{theorem}
\begin{proof}
\textbf{(a)}$\Longrightarrow$ \textbf{(b)} We will prove $\left(\ac^\dag\right)^*\star_M\xc\star_M\left(\ac^\dag\right)^*=\left(\ac^\dag\right)^*$. Suppose $\xc=\ac^-\star_M\ac\star_M\ac^*$. Simple computations give
  \begin{align*}
   \left(\ac^\dag\right)^*\star_M\xc\star_M\left(\ac^\dag\right)^*&=\left(\ac^\dag\right)^*\star_M\ac^-\star_M\ac\star_M\ac^*\star_M\left(\ac^\dag\right)^*\\
   &=\left(\ac^\dag\star_M\ac\star_M\ac^\dag\right)^*\star_M\ac^-\star_M\ac\star_M\ac^*\star_M\left(\ac^\dag\right)^*\\
  &=\left(\ac^\dag\right)^*\star_M\left(\ac^\dag\star_M\ac\right)^*\star_M\ac^-\star_M\ac\star_M\ac^*\star_M\left(\ac^\dag\right)^*\\
  &=\left(\ac^\dag\right)^*\star_M\ac^\dag\star_M\ac\star_M\ac^-\star_M\ac\star_M\ac^*\star_M\left(\ac^\dag\right)^*\\
  &=\left(\ac^\dag\right)^*\star_M\ac^\dag\star_M\ac\star_M\ac^*\star_M\left(\ac^\dag\right)^*\\
  &=\left(\ac^\dag\right)^*\star_M\left(\ac^\dag\star_M\ac\right)^*\star_M\ac^*\star_M\left(\ac^\dag\right)^*\\
  &=\left(\ac^\dag\star_M\ac\star_M\ac^\dag\star_M\ac\star_M\ac^\dag\right)^*\\
  &=\left(\ac^\dag\star_M\ac\star_M\ac^\dag\right)^*\\
  &=\left(\ac^\dag\right)^*.
  \end{align*}
\textbf{(b)}$\Longrightarrow$ \textbf{(c)}  By $\ac\star_M\xc=\ac\star_M\ac^*$, one has
\begin{align*}
 \left(\ac^\dag\right)^*\star_M\xc&=\left(\ac^\dag\star_M\ac\star_M\ac^\dag\right)^*\star_M\xc
 =\left(\ac^\dag\right)^*\star_M\left(\ac^\dag\star_M\ac\right)^*\star_M\xc\\
 &=\left(\ac^\dag\right)^*\star_M\ac^\dag\star_M\ac\star_M\xc
 =\left(\ac^\dag\right)^*\star_M\ac^\dag\star_M\ac\star_M\ac^*\\
 &=\left(\ac^\dag\right)^*\star_M\left(\ac^\dag\star_M\ac\right)^*\star_M\ac^*
 =\left(\ac\star_M\ac^\dag\star_M\ac\star_M\ac^\dag\right)^*\\
 &=\ac\star_M\ac^\dag.
\end{align*}
\textbf{(c)}$\Longrightarrow$ \textbf{(d)} It is obviously true.\\
\textbf{(d)}$\Longrightarrow$ \textbf{(a)} Suppose that $\xc\star_M\left(\ac^\dag\right)^*\star_M\xc=\xc$, $\left(\ac^\dag\right)^*\star_M\xc=\ac\star_M\ac^\dag$ and $\xc\star_M\left(\ac^\dag\right)^*=\ac^-\star_M\ac$, then we have
\begin{align*}
  \xc&=\xc\star_M\left(\ac^\dag\right)^*\star_M\xc=\xc\star_M\left(\ac^\dag\star_M\ac\star_M\ac^\dag\right)^*\star_M\xc=\xc\star_M\left(\ac^\dag\right)^*\star_M\ac^*\star_M\left(\ac^\dag\right)^*\star_M\xc\\
  &=\ac^-\star_M\ac\star_M\ac^*\star_M\ac\star_M\ac^\dag=\ac^-\star_M\ac\star_M\ac^*.
\end{align*}
\textbf{(a)}$\Longrightarrow$ \textbf{(e)} Let $\xc=\ac^-\star_M\ac\star_M\ac^*$. Then, we have
\begin{align*}
  \ac^-\star_M\ac\star_M\xc=\ac^-\star_M\ac\star_M\ac^-\star_M\ac\star_M\ac^*=\ac^-\star_M\ac\star_M\ac^*=\xc
\end{align*}
and
\begin{align*}
  \ac\star_M\xc=\ac\star_M\ac^-\star_M\ac\star_M\ac^*=\ac\star_M\ac^*.
\end{align*}

\end{proof}
\begin{theorem}
Let $\ac\in\mathbb{C}^{n_1\times n_2\times n_3}$ and $\ac^-\in \ac\left\{1\right\}$. Then, the following statements are equivalent.
\begin{description}
  \item  [(a)] $\xc=\ac^-\star_M\ac\star_M\ac^*$,
  \item  [(b)] $\ac^-\star_M\ac\star_M\xc=\xc$, $\ac^\dag\star_M\ac\star_M\xc=\ac^*$,
  \item  [(c)] $\ac^-\star_M\ac\star_M\xc=\xc$, $\left(\ac^\dag\right)^*\star_M\xc=\ac\star_M\ac^\dag$,
 \item  [(d)] $\ac^-\star_M\ac\star_M\xc=\xc$, $\left(\ac^\dag\right)^*\star_M\xc=\ac\star_M\ac^\dag$.
\end{description}

\end{theorem}

\begin{proof}
\textbf{(a)}$\Longrightarrow$ \textbf{(b)} Suppose $\xc=\ac^-\star_M\ac\star_M\ac^*$. Then, one has
\begin{align*}
  \ac^-\star_M\ac\star_M\xc=\ac^-\star_M\ac\star_M\ac^-\star_M\ac\star_M\ac^*=\ac^-\star_M\ac\star_M\ac^*=\xc
\end{align*}
and
\begin{align*}
  \ac\star_M\xc=\ac\star_M\ac^-\star_M\ac\star_M\ac^*=\ac\star_M\ac^*.
\end{align*}
\textbf{(b)}$\Longrightarrow$ \textbf{(c)} If $\ac\star_M\xc=\ac\star_M\ac^*$, then $\ac^\dag\star_M\ac\star_M\xc=\ac^\dag\star_M\ac\star_M\ac^*=\ac^*\star_M\left(\ac^*\right)^\dag\star_M\ac^*=\ac^*$.\\
\textbf{(c)}$\Longrightarrow$ \textbf{(d)} By using $\ac^\dag\star_M\ac\star_M\xc=\ac^*$, one has
\begin{align*}
 \left(\ac^\dag\right)^*\star_M\xc&=\left(\ac^\dag\star_M\ac\star_M\ac^\dag\right)^*\star_M\xc=\left(\ac^\dag\right)^*\star_M\left(\ac^\dag\star_M\ac\right)^*\star_M\xc\\
 &=\left(\ac^\dag\right)^*\star_M\ac^\dag\star_M\ac\star_M\xc=\left(\ac^\dag\right)^*\star_M\ac^*=\left(\ac\star_M\ac^\dag\right)^*\\
 &=\ac\star_M\ac^\dag.
\end{align*}
\textbf{(d)}$\Longrightarrow$ \textbf{(a)} Suppose that $\ac^-\star_M\ac\star_M\xc=\xc$ and $\left(\ac^\dag\right)^*\star_M\xc=\ac\star_M\ac^\dag$. Hence, we can get
\begin{align*}
 \ac^-\star_M\ac\star_M\xc&=\ac^-\star_M\left(\ac\star_M\ac^\dag\star_M\ac\right)\star_M\xc=\ac^-\star_M\ac\star_M\left(\ac^\dag\star_M\ac\right)^*\star_M\xc\\
 &=\ac^-\star_M\ac\star_M\ac^*\star_M\left(\ac^\dag\right)^*\star_M\xc=\ac^-\star_M\ac\star_M\ac^*\star_M\ac\star_M\ac^\dag\\
 &=\ac^-\star_M\ac\star_M\ac^*.
\end{align*}

\end{proof}

\begin{theorem}
Let $\ac\in\mathbb{C}^{n_1\times n_2\times n_3}$ and $\ac^-\in \ac\left\{1\right\}$. Then, the following statements are equivalent.
\begin{description}
  \item  [(a)] $\xc=\ac^-\star_M\ac\star_M\ac^*$,
  \item  [(b)] $\xc\star_M\ac\star_M\ac^\dag=\xc$, $\xc\star_M\left(\ac^\dag\right)^*=\ac^-\star_M\ac$,
  \item  [(c)] $\xc\star_M\ac\star_M\ac^\dag=\xc$, $\xc\star_M\ac=\ac^-\star_M\ac\star_M\ac^*\star_M\ac$,
  \item  [(d)] $\ac^-\star_M\ac\star_M\xc\star_M\ac\star_M\ac^\dag=\xc$, $\ac\star_M\xc\star_M\left(\ac^\dag\right)^*=\ac$.
\end{description}
\end{theorem}

\begin{proof}
\textbf{(a)}$\Longrightarrow$ \textbf{(b)} Let $\xc=\ac^-\star_M\ac\star_M\ac^*$. Then, it follows that $$\xc\star_M\ac\star_M\ac^\dag=\ac^-\star_M\ac\star_M\ac^*\star_M\ac\star_M\ac^\dag=\ac^-\star_M\ac\star_M\ac^*=\xc.$$
\textbf{(b)}$\Longrightarrow$ \textbf{(c)} By using $\xc\star_M\left(\ac^\dag\right)^*=\ac^-\star_M\ac$, one has
\begin{align*}
  \xc\star_M\ac&=\xc\star_M\left(\ac\star_M\ac^\dag\star_M\ac\right)=\xc\star_M\left(\ac\star_M\ac^\dag\right)^*\star_M\ac\\
  &=\xc\star_M\left(\ac^\dag\right)^*\star_M\ac^*\star_M\ac=\ac^-\star_M\ac\star_M\ac^*\star_M\ac.
\end{align*}
\textbf{(c)}$\Longrightarrow$ \textbf{(d)} If $\xc\star_M\ac\star_M\ac^\dag=\xc$ and $\xc\star_M\ac=\ac^-\star_M\ac\star_M\ac^*\star_M\ac$, then we can get
\begin{align*}
 \ac^-\star_M\ac\star_M\xc\star_M\ac\star_M\ac^\dag&=\ac^-\star_M\ac\star_M\ac^-\star_M\ac\star_M\ac^*\star_M\ac\star_M\ac^\dag\\
 &=\ac^-\star_M\ac\star_M\ac^*\star_M\ac\star_M\ac^\dag\\
 &=\ac^-\star_M\ac\star_M\ac^*
\end{align*}
and
\begin{align*}
  \ac\star_M\xc\star_M\left(\ac^\dag\right)^*&=\ac\star_M\xc\star_M\left(\ac^\dag\star_M\ac\star_M\ac^\dag\right)^*=\ac\star_M\xc\star_M\ac\star_M\ac^\dag\star_M\left(\ac^\dag\right)^*\\
  &=\ac\star_M\ac^-\star_M\ac\star_M\ac^*\star_M\ac\star_M\ac^\dag\star_M\left(\ac^\dag\right)^*=\ac\star_M\ac^*\star_M\ac\star_M\ac^\dag\star_M\left(\ac^\dag\right)^*\\
  &=\ac\star_M\ac^*\star_M\left(\ac\star_M\ac^\dag\right)^*\star_M\left(\ac^\dag\right)^*=\ac\star_M\ac^*\star_M\left(\ac^\dag\star_M\ac\star_M\ac^\dag\right)^*\\
  &=\ac\star_M\ac^*\star_M\left(\ac^\dag\right)^*=\ac\star_M\left(\ac^\dag\star_M\ac\right)^*\\
  &=\ac\star_M\ac^\dag\star_M\ac=\ac.
\end{align*}
\textbf{(d)}$\Longrightarrow$ \textbf{(a)} If $\ac^-\star_M\ac\star_M\xc\star_M\ac\star_M\ac^\dag=\xc$ and $\ac\star_M\xc\star_M\left(\ac^\dag\right)^*=\ac$, then we can get
\begin{align*}
  \xc=\ac^-\star_M\ac\star_M\xc\star_M\ac\star_M\ac^\dag=\ac^-\star_M\ac\star_M\xc\star_M\left(\ac^\dag\right)^*\star_M\ac^*=\ac^-\star_M\ac\star_M\ac^*.
\end{align*}

\end{proof}

\begin{theorem}
Let $\ac\in\mathbb{C}^{n_1\times n_2\times n_3}$ and $\ac^-\in \ac\left\{1\right\}$. Then, the following statements are equivalent.
\begin{description}
  \item  [(a)] $\ac\star_M\ac^*\star_M\ac=\ac$,
  \item  [(b)] $\ac\star_M\ac^{-,*}\star_M\ac=\ac$,
  \item  [(c)] $\ac^{-,*}\star_M\ac\star_M\ac^{-,*}=\ac^{-,*}$.
\end{description}
\begin{proof}
\textbf{(a)}$\Longrightarrow$ \textbf{(b)} If $\ac\star_M\ac^*\star_M\ac=\ac$, then $$\ac\star_M\ac^{-,*}\star_M\ac=\ac\star_M\ac^-\star_M\ac\star_M\ac^*\star_M\ac=\ac\star_M\ac^-\star_M\ac=\ac.$$
\textbf{(b)}$\Longrightarrow$ \textbf{(a)} If $\ac\star_M\ac^{-,*}\star_M\ac=\ac$, then
\begin{align*}
\ac= \ac\star_M\ac^{-,*}\star_M\ac
     =\ac\star_M\ac^-\star_M\ac\star_M\ac^*\star_M\ac
     =\ac\star_M\ac^*\star_M\ac.
\end{align*}
\textbf{(a)}$\Longrightarrow$ \textbf{(c)} Since $\ac\star_M\ac^*\star_M\ac=\ac$, one has
\begin{align*}
  \ac^{-,*}\star_M\ac\star_M\ac^{-,*}&=\ac^-\star_M\ac\star_M\ac^*\star_M\ac\star_M\ac^-\star_M\ac\star_M\ac^*\\
  &=\ac^-\star_M\ac\star_M\ac^-\star_M\ac\star_M\ac^*\\
  &=\ac^-\star_M\ac\star_M\ac^*=\ac^{-,*}.
\end{align*}
\textbf{(c)}$\Longrightarrow$ \textbf{(a)} If $\ac^{-,*}\star_M\ac\star_M\ac^{-,*}=\ac^{-,*}$, then
\begin{align*}
\ac^-\star_M\ac\star_M\ac^*&=\ac^{-,*}=\ac^{-,*}\star_M\ac\star_M\ac^{-,*}=\ac^-\star_M\ac\star_M\ac^*\star_M\ac\star_M\ac^-\star_M\ac\star_M\ac^*\\
 &= \ac^-\star_M\ac\star_M\ac^*\star_M\ac\star_M\ac^*.
\end{align*}
Multiplying $\ac$ by  both sides of this equation from the left gives $$\ac\star_M\ac^*=\ac\star_M\ac^*\star_M\ac\star_M\ac^*$$
 and then multiplying $\left(\ac^\dag\right)^*$ by both sides of this equation from the right yields $\ac=\ac\star_M\ac^*\star_M\ac$.

\end{proof}
\end{theorem}

\begin{theorem}
Let $\ac\in\mathbb{C}^{n_1\times n_2\times n_3}$. Then,
$$\ac\left\{-,*\right\}\star_M\left(\ac^\dag\right)^*\star_M\ac\left\{-,*\right\}\subseteq\ac\left\{-,*\right\}.$$
\end{theorem}
\begin{proof}
Assume that $\uc,~\vc\in\ac\left\{-,*\right\}$, and let $\uc=\ac^-\star_M\ac\star_M\ac^*$,  $\vc=\ac^{\left(1\right)}\star_M\ac\star_M\ac^\dag$, for some $\ac^-,~\ac^{\left(1\right)}\in\ac\left\{1\right\}$. Denote $\wc=\uc\star_M\left(\ac^\dag\right)^*\star_M\vc$. Then,
\begin{align*}
  \wc&=\ac^-\star_M\ac\star_M\ac^*\star_M\left(\ac^\dag\right)^*\star_M\ac^{\left(1\right)}\star_M\ac\star_M\ac^*\\
  &=\ac^-\star_M\ac\star_M\ac^\dag\star_M\ac\star_M\ac^{\left(1\right)}\star_M\ac\star_M\ac^*\\
  &=\ac^-\star_M\ac\star_M\ac^{\left(1\right)}\star_M\ac\star_M\ac^*\\
  &=\ac^-\star_M\ac\star_M\ac^*\in\ac\left\{-,*\right\}.
\end{align*}
So, $\ac\left\{-,*\right\}\star_M\left(\ac^\dag\right)^*\star_M\ac\left\{-,*\right\}\subseteq\ac\left\{-,*\right\}$.
\end{proof}
\section{Applications to multilinear systems}

Solutions of multilinear systems in terms of the 1-MP inverse, the 1-D inverse and the 1-Star inverse are given as an application in this part. The application of the 1-MP inverse multilinear system is given as follows.

\begin{theorem}
Let $\ac\in\mathbb{C}^{n_1\times n_2\times n_3}$ and  $\bc\in\mathbb{C}^{n_1\times 1\times n_3}$.  Then, for arbitrary $\mathcal{Z}\in\mathbb{C}^{n_2\times 1\times n_3}$,
\begin{equation}\label{ddgg}
\ac\star_M\xc=\ac\star_M\ac^{\dag}\star_M\bc,
\end{equation}
is solvable and $\xc=\ac^{-,\dag}\star_M\bc+\left(\mathcal{I}-\ac^{-}\star_M\ac\right)\star_M\mathcal{Z}$ are the solutions of (\ref{ddgg}).
\begin{proof}
Clearly, $\ac^{-,\dag}\star_M\bc$ satisfies (\ref{ddgg}). Hence, the system (\ref{ddgg}) is consistent. Now,
$$\ac\star_M\xc
=\ac\star_M\ac^{-,\dag}
\star_M\bc+\ac\star_M\left(\mathcal{I}-\ac^{-}
\star_M\ac\right)\star_M\mathcal{Z}=\ac\star_M\ac^{\dag}\star_M\bc.$$
Also, if $\xc$ is a solution of (\ref{ddgg}), then
$$\xc=\ac^{-,\dag}\star_M\bc+\left(\mathcal{I}-\ac^{-}\star_M\ac\right)\star_M\xc.$$
So, the general solution of the system (\ref{ddgg}) is given by
$$\xc=\ac^{-,\dag}\star_M\bc+\left(\mathcal{I}-\ac^{-}\star_M\ac\right)\star_M\mathcal{Z}.$$
\end{proof}

\end{theorem}

\begin{theorem}
Let $\ac\in\mathbb{C}^{n_1\times n_2\times n_3}$ and $\ac^-$ be a $\{1\}$-inverse of $\ac$, then for some $\zc\in\mathbb{C}^{n_2\times n_1\times n_3}$,
\begin{equation}\label{zzhh}
\ac^{-,\dag}=\xc\star_M\ac\star_M\ac^{\dag},
\end{equation}
is solvable and $\xc=\ac^{-}+\zc\star_M\left(\mathcal {I}-\ac\star_M\ac^{\dag}\right)$ is one  of the solutions of (\ref{zzhh}).
\begin{proof}
Suppose $\xc=\ac^{-}+\zc\star_M\left(\mathcal {I}-\ac\star_M\ac^{\dag}\right)$, then
\begin{align*}
 \xc\star_M\ac\star_M\ac^\dag&=\ac^{-}\star_M\ac\star_M\ac^{\dag}+\zc\star_M\left(\mathcal {I}-\ac\star_M\ac^{\dag}\right)\star_M\ac\star_M\ac^\dag\\
&=\ac^-\star_M\ac\star_M\ac^{\dag}=\ac^{-,\dag}.
\end{align*}
On the contrary, let $\ac^{-,\dag}=\xc\star_M\ac\star_M\ac^{\dag}$. Obviously, $\ac^-$ is a particular solution of $\xc\star_M\ac\star_M\ac^{\dag}=\ac^{-,\dag}$. If $\zc$ is any solution of
$\xc\star_M\ac\star_M\ac^{\dag}=\mathcal {O}$, then $\zc\star_M\ac\star_M\ac^\dag=\mathcal {O}$. Thus we can express $\zc$ as $$\zc=\zc-\zc\star_M\ac\star_M\ac^\dag=\zc\star_M\left(\mathcal {I}-\ac\star_M\ac^\dag\right).$$
Hence the general solution of $\xc\star_M\ac\star_M\ac^\dag=\mathcal {O}$ is given by $\xc=\zc\star_M\left(\mathcal {I}-\ac\star_M\ac^\dag\right)$. Consequently, $$\xc=\ac^-+\zc\star_M\left(\mathcal {I}-\ac\star_M\ac^\dag\right)$$ is the general solution of $\xc\star_M\ac\star_M\ac^{\dag}=\ac^{-,\dag}$.
\end{proof}
\end{theorem}

Next, the application of the multilinear system related to the 1-D inverse is introduced.

\begin{theorem}
Let $\ac\in\mathbb{C}^{n_1\times n_1\times n_3}$ and  $\bc\in\mathbb{C}^{n_1\times 1\times n_3}$. Suppose  $\ind\left(\ac\right)=k$.  Then, for arbitrary $\mathcal{Z}\in\mathbb{C}^{n_1\times 1\times n_3}$,
\begin{equation}\label{ddyy}
\ac\star_M\xc=\ac\star_M\ac^D\star_M\bc,
\end{equation}
is solvable and $\xc=\ac^{-,\ D}\star_M\bc+\left(\mathcal{I}-\ac^{-}\star_M\ac\right)\star_M\mathcal{Z}$ are some of the solutions of (\ref{ddyy}).
\begin{proof}
Obviously, $\ac^{-}\star_M\ac\star_M\ac^D\star_M\bc=\ac^{-,\ D}\star_M\bc$ satisfies (\ref{ddyy}). Therefore, the system (\ref{ddyy}) is solvable. Now,
$$\ac\star_M\xc
=\ac\star_M\ac^{-,\ D}
\star_M\bc+\ac\star_M\left(\mathcal{I}-\ac^{-}
\star_M\ac\right)\star_M\mathcal{Z}=\ac\star_M\ac^D\star_M\bc.$$
Meanwhile, if $\xc$ is a solution of (\ref{ddyy}), then
$$\xc=\ac^{-,\ D}\star_M\bc+\left(\mathcal{I}-\ac^{-}\star_M\ac\right)\star_M\xc.$$
Hence, the general solution of the system (\ref{ddyy}) is given by
$$\xc=\ac^{-,\ D}\star_M\bc+\left(\mathcal{I}-\ac^{-}\star_M\ac\right)\star_M\mathcal{Z}.$$
\end{proof}

\end{theorem}

\begin{theorem}
Let $\ac\in\mathbb{C}^{n_1\times n_1\times n_3}$, ind$\left(\ac\right)=k$ and $\ac^-$ be a fixed $\left\{1\right\}$-inverse of $\ac$, then for some $\zc\in\mathbb{C}^{n_1\times n_1\times n_3}$,
\begin{equation}\label{zzgg}
\ac^{-,\ D}=\xc\star_M\ac\star_M\ac^{D},
\end{equation}
is solvable and $\xc=\ac^{-}+\mathcal {Z}\star_M\left(\mathcal {I}-\ac\star_M\ac^{D}\right)$ is one  of the solutions of (\ref{zzgg}).
\begin{proof}
If $\xc=\ac^{-}+\mathcal {Z}\star_M\left(\mathcal {I}-\ac\star_M\ac^{D}\right)$, then
\begin{align*}
 \xc\star_M\ac\star_M\ac^D&=\ac^{-}\star_M\ac\star_M\ac^D+\zc\star_M\left(\mathcal {I}-\ac\star_M\ac^{D}\right)\star_M\ac\star_M\ac^D\\
&=\ac^-\star_M\ac\star_M\ac^D=\ac^{-,\ D}.
\end{align*}
Suppose $\ac^{-,\ D}=\xc\star_M\ac\star_M\ac^{D}$. Obviously, $\ac^-$ is a particular solution of $\xc\star_M\ac\star_M\ac^{D}=\ac^{-,\ D}$. If $\zc$ is any solution of
$\xc\star_M\ac\star_M\ac^{D}=\mathcal {O}$, then $\zc\star_M\ac\star_M\ac^D=\mathcal {O}$. Thus, one can express $\zc$ as $$\zc=\zc-\zc\star_M\ac\star_M\ac^D=\zc\star_M\left(\mathcal {I}-\ac\star_M\ac^D\right).$$
Therefore, the general solution of $\xc\star_M\ac\star_M\ac^D=\mathcal {O}$ is given by $\xc=\zc\star_M\left(\mathcal {I}-\ac\star_M\ac^D\right)$. Consequently, $$\xc=\ac^-+\zc\star_M\left(\mathcal {I}-\ac\star_M\ac^D\right)$$ is the general solution of $\xc\star_M\ac\star_M\ac^{D}=\ac^{-,\ D}$.
\end{proof}
\end{theorem}

Finally, the application of the multilinear system related to the 1-Star inverse is introduced.

\begin{theorem}\label{to}
Let $\ac\in\mathbb{C}^{n_1\times n_2\times n_3}$ and  $\bc\in\mathbb{C}^{n_1\times 1\times n_3}$.   Then, for arbitrary $\mathcal{Z}\in\mathbb{C}^{n_2\times 1\times n_3}$,
\begin{equation}\label{ddxx}
\ac\star_M\xc=\ac\star_M\ac^*\star_M\bc,
\end{equation}
is solvable and $\xc=\ac^{-,*}\star_M\bc+\left(\mathcal{I}-\ac^{-}\star_M\ac\right)\star_M\mathcal{Z}$ are some of the solutions of (\ref{ddxx}).
\begin{proof}
Obviously, $\ac^{-}\star_M\ac\star_M\ac^*\star_M\bc$ satisfies (\ref{ddxx}). So, the system (\ref{ddxx}) is consistent. Now,
$$\ac\star_M\xc
=\ac\star_M\ac^{-,*}
\star_M\bc+\ac\star_M\left(\mathcal{I}-\ac^{-}
\star_M\ac\right)\star_M\mathcal{Z}=\ac\star_M\ac^*\star_M\bc.$$
Also, if $\xc$ is a solution of (\ref{ddxx}), then
$$\xc=\ac^{-,*}\star_M\bc+\left(\mathcal{I}-\ac^{-}\star_M\ac\right)\star_M\xc.$$
So, the general solution of the system (\ref{ddxx}) is given by
$$\xc=\ac^{-,*}\star_M\bc+\left(\mathcal{I}-\ac^{-}\star_M\ac\right)\star_M\mathcal{Z}.$$
\end{proof}

\end{theorem}
\begin{example}
Let $\ac\in\mathbb{C}^{2\times3\times3},~\bc\in\mathbb{C}^{2\times 1\times3},~M\in\mathbb{C}^{3\times3},~\mathcal{Z}\in\mathbb{C}^{3\times1\times3}$ with entries
$$\ac^{\left(1\right)}=
\begin{bmatrix}
0&0&1\\
1&-1&1
\end{bmatrix},~
\ac^{\left(2\right)}=
\begin{bmatrix}
1&0&0\\
0&0&1
\end{bmatrix},~
\ac^{\left(3\right)}=
\begin{bmatrix}
0&0&0\\
0&1&-1
\end{bmatrix},
$$
$$
\mathcal B^{\left(1\right)}=
\begin{bmatrix}
0\\
2
\end{bmatrix},~
\mathcal B^{\left(2\right)}=
\begin{bmatrix}
0\\
1
\end{bmatrix},~
\mathcal B^{\left(3\right)}=
\begin{bmatrix}
1\\
-1
\end{bmatrix},$$ $$\mathcal Z^{\left(1\right)}=
\begin{bmatrix}
z_{11}\\
z_{12}\\
z_{13}
\end{bmatrix},~
\mathcal Z^{\left(2\right)}=
\begin{bmatrix}
z_{21}\\
z_{22}\\
z_{23}
\end{bmatrix},~
\mathcal Z^{\left(3\right)}=
\begin{bmatrix}
z_{31}\\
z_{32}\\
z_{33}
\end{bmatrix}, \
M=
\begin{bmatrix}
1 & 0 & 1\\
0 & 1 & 0 \\
0 & 1 & 1
\end{bmatrix}.
$$
Firstly, we get the $\widehat{\ac^-}$ and 1-Star inverses by using Algorithm \ref{a3}, that is,
$$
\left(\widehat{\ac^-}\right)^{\left(1\right)}=
\begin{bmatrix}
0       & 1\\
a_{21}  & a_{22}\\
1       & 0
\end{bmatrix},~
\left(\widehat{\ac^-}\right)^{\left(2\right)}=
\begin{bmatrix}
1       & 0\\
b_{21}  & b_{22}\\
1       & 0
\end{bmatrix},~
\left(\widehat{\ac^-}\right)^{\left(3\right)}=
\begin{bmatrix}
1      & 0\\
0      & 1\\
c_{31} & c_{32}
\end{bmatrix},
$$
$$
\left(\ac^{-,*}\right)^{\left(1\right)}=
\begin{bmatrix}
0              &1               \\
a_{21}+b_{21}  & a_{22}+b_{22}-1 \\
1-c_{31}       &   1-c_{32}
\end{bmatrix},~
\left(\ac^{-,*}\right)^{\left(2\right)}=
\begin{bmatrix}
1     & 0 \\
b_{21}& b_{22}\\
   0  & 1
\end{bmatrix},~
\left(\ac^{-,*}\right)^{\left(3\right)}=
\begin{bmatrix}
0       & 0                 \\
-b_{21} & 1-b_{22} \\
c_{31}  & c_{32}-1
\end{bmatrix}.
$$
Then, we calculate $\xc=\ac^{-,*}\star_M\bc+\left(\mathcal{I}-\ac^{-}\star_M\ac\right)\star_M\mathcal{Z}$ by Theorem \ref{to}, which is
$$
\mathcal X^{\left(1\right)}=\begin{bmatrix}
0  \\
a_{21}+a_{22}+b_{22}+z_{12}+z_{22}+z_{32}-a_{22}\left(z_{11}+z_{31}\right)-a_{21}\left(z_{13}+z_{33}\right)-b_{21}z_{21}-b_{22}z_{23} \\
c_{31}\left(z_{21}+z_{31}\right)-z_{23}-z_{33}-c_{31}+c_{32}\left(z_{22}+z_{32}\right)+2
\end{bmatrix},$$
$$\mathcal X^{\left(2\right)}=\begin{bmatrix}0\\b_{22}+z_{22}-b_{21}z_{21}-b_{22}z_{23}\\
1\end{bmatrix},$$
$$\mathcal X^{\left(3\right)}=
\begin{bmatrix}
1\\
b_{21}z_{21}-z_{22}-b_{22}+b_{22}z_{23}\\
c_{31}+z_{23}+z_{33}-c_{31}\left(z_{21}+z_{31}\right)-c_{32}\left(z_{22}+z_{32}\right)-1
\end{bmatrix}.
$$
\end{example}

\section{Conclusion}

This work presents and studies the 1-$\Gamma$ inverse under the M-product. Firstly, the singular value decomposition is used to present an expression for the 1-MP inverse. An equivalent numerical algorithm for computing this inverse is developed. Next, some 1-MP inverse characterization is shown. Secondly, the 1-D inverse of the tensor is further explained. A partial description of the 1-D inverse is provided. Thirdly, the 1-Star inverse under the M-product operation is described. The 1-Star inverse's numerical computation technique is also constructed, and the associated attributes are provided. Based on these work, the solutions to the 1-MP inverses, 1-D inverses, and 1-Star inverses of the multilinear equations are provided, along with an example program that computes the corresponding 1-$\Gamma$ inverse.

~

{\bf\large Funding}

~

This work was supported by "The Special Fund for Science and Technological Bases and Talents of Guangxi" (No. GUIKE AD21220024).

~

{\bf\large Declarations}

~

\textbf{Conflict of interest} No potential conflict of interest was reported by the authors.

~

\end{spacing}


\begin{thebibliography}{99}

\bibitem{kong1}
H. Kong, X. Xie, Z. Lin. $t$-Schatten-$p$ norm for low-rank tensor recovery. IEEE Journal of Selected Topics in Signal Processing, 12 (2018), pp. 1405-1419.

\bibitem{L3}
Y. Liu, L. Chen, C. Zhu. Improved robust tensor principal component analysis via low-rank core matrix. IEEE Journal of Selected Topics in Signal Processing, 12 (2018), pp. 1378-1389.

\bibitem{Che}
M. Che, X. Wang, Y. Wei, X. Zhao,
Fast randomized tensor singular value thresholding for low-rank tensor optimization, Numer. Linear Algebra Appl., 29 (6) (2022), e2444.

\bibitem{QW}
Q. Wang, X. Xu, Iterative algorithms for solving some tensor equations. Linear and Multilinear Algebra, 67 (2019), pp. 1325-1349.

\bibitem{KBHH}
M. Kilmer, K. Braman, N. Hao, R. Hoover, Third-order tensors as operators on matrices: a theoretical and computational framework with applications
in imaging, SIAM J. Matrix Anal. Appl., 34 (2013), pp. 148-172.

\bibitem{MSL}
C. Martin, R. Shafer, B. LaRue,
An Order-$p$ Tensor Factorization with Applications in Imaging, SIAM J. Scientific Computing, 35 (2013),
pp. 474-490.

\bibitem{Soltani}
S. Soltani, M. Kilmer, P. Hansen, A tensor-based dictionary learning approach to tomographic image reconstruction, BIT Numer. Math., 56 (2016), pp. 1425-1454.

\bibitem{Tarzanagh}
D. Tarzanagh, G. Michailidis, Fast randomized algorithms for t-product based tensor operations and decompositions with applications to imaging data, SIAM J. Imaging Sci., 11 (2018), pp. 2629-2664.

\bibitem{4}
N. Hao, M. Kilmer, K. Braman, R. Hoover, Facial recognition using tensor-tensor decompositions, SIAM J. Imaging Sci., 6 (2013), pp. 437-463.

\bibitem{5}
W. Hu, Y. Yang, W. Zhang, Y. Xie, Moving object detection using tensor-based low-rank and saliently fused-sparse decomposition, IEEE Trans. Image Process., 26 (2017), pp. 724-737.

\bibitem{1}
T. Chan, T. Yang, Polar $n$-complex and $n$-bicomplex singular value decomposition and principal component pursuit, IEEE Trans. Signal Process., 64 (2016), pp. 6533-6544.

\bibitem{RU} L. Reichel, O. Ugwu,
Tensor Krylov subspace methods with an invertible linear transform product applied to image processing, Appl. Numer. Math., 166 (2021), pp. 186-207.


\bibitem{2}
T. Liu, L. Chen, C. Zhu, Improved robust tensor principal component analysis via low-rank core matrix, IEEE J. Sel. Top. Signal Process., 12 (2018), pp. 1378-1389.

\bibitem{3}
Z. Long, Y. Liu, L. Chen, C. Zhu, Low rank tensor completion for multiway visual data, Signal Process., 155 (2019), pp. 301-316.

\bibitem{Chen} J. Chen, W. Ma, Y. Miao, Y. Wei,
Perturbations of Tensor-Schur decomposition and its applications to multilinear control systems and facial recognitions, Neurocomputing, 547 (2023), 126359.

\bibitem{Hu}
W. Hu, Y. Yang, W. Zhang, Y. Xie. Moving object detection using tensor-based low-rank and saliently fused-sparse decomposition. IEEE Transactions on Image Processing, 26 (2016), pp. 724-737.

\bibitem{Long}
Z. Long, Y. Liu, L. Chen, C. Zhu. Low rank tensor completion for multiway visual data. Signal processing, 155 (2019), pp. 301-316.

\bibitem{A2}
A. Wang, Z. Lai, Z. Jin, Noisy low-tubal-rank tensor completion, Neurocomputing 330 (2019), pp. 267-279.



\bibitem{K1} E. Kilmer, C. Martin. Factorization strategies for third-order tensors. Linear Algebra and its Applications, 435 (2011), pp. 641-658.

\bibitem{J1}
X. Jin. Developments and applications of block Toeplitz iterative solvers. Springer Science  Business Media, 2003.

\bibitem{KK}
E. Kernfeld, M. Kilmer, S. Aeron,
Tensor-tensor products with invertible linear transforms, Linear Algebra Appl, 485 (2015), pp. 545-570.

\bibitem{A1}
E. Albert, W. Perrett, G. Jeffery. The foundation of the general theory of relativity. Annalen der Physik, 354 (1916), 769.

\bibitem{Shao1}
J. Shao. A general product of tensors with applications. Linear Algebra and its applications, 439 (2013), pp. 2350-2366.

\bibitem{L1}
L. Qi, Z. Luo. Tensor analysis: spectral theory and special tensors. Society for Industrial and Applied Mathematics, 2017.

\bibitem{Perrone}
M E. Kilmer, C D. Martin, L. Perrone.: A third-order generalization of the matrix SVD as a product
of third-order tensors. Technical Report 2008-4, Tufts University (2008).

\bibitem{Lund}
K Lund. The tensor t-function: A definition for functions of third-order tensors. Numerical Linear Algebra with Applications, 27 (3) (2020), pp. e2288.

\bibitem{Newman}
E. Newman, L. Horesh, H. Avron, M. Kilmer, Stable tensor neural networks for rapid deep learning,
arXiv preprint, arXiv :1811 .06569, 2018.


\bibitem{Miao2}
Y. Miao, L. Qi, Y. Wei, T-Jordan Canonical Form and T-Drazin Inverse based on the T-Product, Com. Appl. Math. Comput., 3 (2021), pp. 201-220.

\bibitem{Sun}
L. Sun, B. Zheng, C. Bu, Y. Wei, Moore-Penrose inverse of tensors via Einstein product, Linear
Multilinear Algebra, 64 (4) (2016), pp. 686-698.

\bibitem{Ji2}
J. Ji, Y. Wei, The Drazin inverse of an even-order tensor and its application to singular tensor equations,
Comput. Math. Appl., 75 (9) (2018), pp. 3402-3413.

\bibitem{Nandi}
R. Behera, A. Nandi, J. Sahoo, Further results on the Drazin inverse of even order tensors, Numer Linear Algebr,
27 (5) (2020), e2317.

\bibitem{Liang}
M. Liang, B. Zheng, Further results on Moore-Penrose inverses of tensors with application to tensor
nearness problems, Comput. Math. Appl., 77 (5) (2019), pp. 1282-1293.

\bibitem{Che1} M. Che, Y. Wei,
An efficient algorithm for computing the approximate t-URV and its applications, Journal of Scientific Computing, 92 (3) (2022), pp. 557-583.

\bibitem{Cong} Z. Cong, H. Ma,
Acute perturbation for Moore-Penrose inverses of tensors via the T-product, J. Appl. Math. Comput., 68 (6) (2022), pp. 3799-3820.

\bibitem{Behera}
R. Behera, D. Mishra, Further results on generalized inverses of tensors via the Einstein product, Linear Multilinear Algebra, 65 (8) (2017), pp. 1662-1682.


\bibitem{Cong2} Z. Cong, H. Ma,
Characterizations and perturbations of the core-EP inverse of tensors based on the T-product, Numer. Funct. Anal. Optim., 43 (10) (2022), pp. 1150-1200.

\bibitem{Sahoo} J. Sahoo, R. Behera, P. S. Stanimirovi\'c, V. Katsikis, H. Ma,
Core and core-EP inverses of tensors, Comput. Appl. Math., 39 (1) (2020), pp. 1-28.

\bibitem{JXWL} H. Jin, S. Xu, Y. Wang, X. Liu. The Moore-Penrose inverse of tensors via the M-Product. Computational and Applied Mathematics. 42 (6) (2023), 294.

\bibitem{SPBS}
J. Sahoo, S. Panda, R. Behera, P. S. Stanimirovi\'{c}, Computation of tensors generalized inverses under M-product and applications, arXiv:2405.16111, 2024.

\bibitem{Panigrahy}
K. Panigrahy, B. Karmakar, J. K. Sahoo, et al. Computation of $ M $-QDR decomposition of tensors and applications, arXiv preprint arXiv:2409.08743, 2024.

\bibitem{KB} T. Kolda, B.  Bader,
Tensor decompositions and applications, SIAM review, 51 (2009), pp. 455-500.





\bibitem{K3} M. Kilmer, L. Horesh, H. Avron, E. Newman,
Tensor-tensor products for optimal representation and compression, arXiv:2001.00046, 2020.


\bibitem{A} A. Ben-Israel, T.N.E. Greville, Generalized Inverses: Theory and Applications, second ed., Springer Verlag, New York, 2003.



































































\end{thebibliography}
\end{document}